\documentclass[a4paper,11pt]{article}

\usepackage{amsmath,amsthm,amsfonts,cite,amsbsy,amscd,amssymb,graphicx,
epsfig,color}

\usepackage[utf8]{inputenc}
\newtheorem{thm}{Theorem}[section]
\newtheorem{cor}[thm]{Corollary}
\newtheorem{lem}[thm]{Lemma}
\newtheorem{rem}[thm]{Remark}
\newtheorem{exa}[thm]{Example}

\newtheorem{prop}[thm]{Proposition}

\begin{document}

\title{Neumann boundary optimal control problems governed by parabolic variational equalities}

\author{Carolina M. Bollo \thanks{Depto. Matem\'atica, FCEFQyN, Universidad Nac. de R\'io Cuarto, Ruta 36 Km 601, 5800 R\'io Cuarto, Argentina.
  E-mail: cbollo@exa.unrc.edu.ar; cgariboldi@exa.unrc.edu.ar}\,\,  Claudia M. Gariboldi $^\ast$\,Domingo A.
Tarzia \thanks{Depto. de Matem\'atica-CONICET, FCE, Universidad Austral, Paraguay 1950, S2000FZF Rosario, Argentina. E-mail:
DTarzia@austral.edu.ar} }

\maketitle

\begin{abstract}
We consider a heat conduction problem $S$ with mixed boundary
conditions in a $n$-dimensional domain $\Omega$ with regular
boundary and a family of problems $S_{\alpha}$ with also mixed boundary conditions in $\Omega$, where $\alpha>0$ is
the heat transfer coefficient on the portion of the boundary
$\Gamma_{1}$. In relation to these state systems, we formulate
\emph{Neumann boundary} optimal control problems on the heat flux $q$ which
is definite on the complementary portion $\Gamma_{2}$ of the
boundary of $\Omega$. We obtain existence and uniqueness of the
optimal controls, the first order optimality conditions in terms of
the adjoint state and the convergence of the optimal controls, the
system state and the adjoint state when the heat transfer
coefficient $\alpha$ goes to infinity. Furthermore, we formulate
particular \emph{boundary} optimal control problems on a real
parameter $\lambda$, in relation to the parabolic problems $S$ and $S_{\alpha}$ and to mixed elliptic problems $P$ and $P_{\alpha}$.  We find a explicit form for the optimal controls, we prove monotony properties and we obtain convergence results when the parameter time goes to infinity.
\end{abstract}

\textbf{keywords:} Parabolic variational equalities, Optimal
control, Mixed boundary conditions, Optimality conditions. Convergence.


{\thispagestyle{empty}} 

\section{Introduction}

Following \cite{GT1,MT,TaBoGa}, we will study some Neumann boundary parabolic and elliptic optimal control problems. We consider a bounded domain $\Omega $ in ${\Bbb R}^{n}$, whose
regular boundary $\Gamma $ consists of the union of the two disjoint
portions $\Gamma _{1}$ and $\Gamma _{2}$ with $|\Gamma_{1}|>0$
and $|\Gamma_{2}|>0$. We denote with $|\Gamma_i|=meas(\Gamma_i)$ (for
$i=1,2$), the $(n-1)$-dimensional Hausdorff measure of the portion $\Gamma_i$ on $\Gamma$. Let
$[0,T]$ a time interval, for a $T>0$. We present the following
heat conduction problems $S$ and $S_{\alpha }$ (for
each parameter $\alpha >0)$ respectively, with mixed boundary
conditions (we denote by $u(t)$ to the function $u(\cdot,t)$):
\begin{equation}
\frac{\partial u}{\partial t}-\Delta u=g\,\ \text{ in }\Omega \ \
\,\,\,\,\,\,u\big|_{\Gamma _{1}}=b\,\,\,\,\,\,\,\,-\frac{\partial
u}{\partial n}\bigg|_{\Gamma _{2}}=q\,\,\,\,\,\,\,\,u(0)=v_b
\label{P}
\end{equation}
\begin{equation}
\frac{\partial u}{\partial t}-\Delta u=g\,\ \text{ in }\Omega \ \ \,\,\,\,\,\,-\frac{\partial u}{\partial n%
}\bigg|_{\Gamma _{1}}=\alpha (u-b)\,\,\,\,\,\,\,\,-\frac{\partial
u}{\partial n}\bigg|_{\Gamma _{2}}=q \,\,\,\,\,\,\,\,u(0)=v_b
\label{Palfa}
\end{equation}
where $u$ is the temperature in $\Omega\times (0,T)$, $g$ is the
internal energy in $\Omega $, $b$ is the temperature on $
\Gamma_{1}$ for (\ref{P}) and the temperature of the external
neighborhood of $\Gamma_{1}$ for (\ref{Palfa}), $v_{b}=b$ on
$\Gamma_{1}$, $q$ is the heat flux on $\Gamma_{2}$ and $\alpha
>0$ is the heat transfer coefficient on $\Gamma_{1}$ through a Robin condition, which satisfy the hypothesis: $g\in
\mathcal{H}=L^2(0,T;L^2(\Omega))$, $q\in
\mathcal{Q}=L^2(0,T;L^2(\Gamma_2))$ and $b\in
H^{\frac{1}{2}}(\Gamma_1)$. In addition, $v_{b}\in H^{1}(\Omega)$ is the initial temperature for (\ref{P}) and (\ref{Palfa}), respectively.

 Let $u$ and
$u_{{\alpha}}$ the unique solutions to the parabolic problems
(\ref{P}) and (\ref{Palfa}), whose variational formulations are
given by (\cite{MT}):
\begin{equation}\label{Pvariacional}
\left\{
\begin{array}{l l}
u-v_b \in L^2(0,T; V_0), \qquad u(0)=v_b\quad\text{and}\quad \dot{u}\in L^2(0,T; V_0')\\
\text{such that}\quad \langle \dot{u}(t), v\rangle+a(u(t),v)=L(t,v),
\quad \forall v\in V_0,
\end{array}
\right.
\end{equation}
\begin{equation}\label{Palfavariacional}
\left\{
\begin{array}{l l}
u_{\alpha} \in L^2(0,T; V), \qquad u_{\alpha}(0)=v_b\quad\text{and}\quad \dot{u}_{\alpha}\in L^2(0,T; V')\\
\text{such that}\quad \langle \dot{u}_{\alpha}(t),
v\rangle+a_{\alpha}(u_{\alpha}(t),v)=L_{\alpha}(t,v), \quad \forall
v\in V,
\end{array}
\right.
\end{equation}
where $\langle\cdot,\cdot\rangle$ denote the duality between the
functional space ($V$ or $V_{0}$) and its dual space ($V'$ or
$V'_{0}$) and
\[
V=H^{1}(\Omega )\,;\,\,\,\,\,\,\,\,\,V_{0}=\{v\in V:\,v\big|_{\Gamma
_{1}}=0\}\,;\,\,\,\,\,\,\,\,\,Q=L^2(\Gamma_2); \quad H=L^2(\Omega)
\]
\vspace{-0.6cm}
\[
(g,h)_{H}=\int_{\Omega}gh\,dx; \quad (q,\eta)_{Q}=\int_{\Gamma
_{2}}q\eta \,d\gamma;
\]
\[
a(u,v)=\int_{\Omega }\nabla u\nabla vdx;\quad a_{\alpha
}(u,v)=a(u,v)+\alpha \int\limits_{\Gamma_1}u vd\gamma\,\,
\]
\vspace{-0.6cm}
\[
L(t,v)=(g(t),v)_{H}-(q(t),v)_{Q};\quad L_{{\alpha}}(t,v)=\,L
(t,v)+\alpha \int\limits_{\Gamma_1}b vd\gamma.
\]
All data, $g$, $q$, $b$, $v_b$ and the domain $\Omega$ with the boundary $\partial \Omega=\Gamma_{1}\cup \Gamma_{2}$ are assumed to be sufficiently
smooth so that the problems (\ref{P}) and (\ref{Palfa}) admit variational solutions in Sobolev spaces. The existence and uniqueness of
the solutions to the variational equalities (\ref{Pvariacional}) and (\ref{Palfavariacional}), is well known, see for example \cite{B,CH,DL,GuH}.

Let $\mathcal{H}=L^2(0,T;H)$ be, with norm
$||.||_{\mathcal{H}}$ and internal product $
(g,h)_{\mathcal{H}}=\int\limits_0^T(g(t),h(t))_{H} dt$, and the space $\mathcal{Q}=L^2(0,T;Q)$, with norm $||.||_{\mathcal{Q}}$ and internal product $(q,\eta)_{\mathcal{Q}}=\int\limits_0^T(q(t),\eta(t))_{Q} dt$.
\newline
For the sake of simplicity, for a Banach space $X$ and $1\leq
p\leq\infty$, we will often use $L^{p}(X)$ instead of
$L^{p}(0,T;X)$.

If we denote by $u_{q}$ and $u_{\alpha q}$ the unique solution to
the problems (\ref{Pvariacional}) and (\ref{Palfavariacional})
respectively, we formulate the following \emph{boundary} optimal
control problems for the heat flux $q$ as control variable, \cite{GT1,Li,Tro}:
\begingroup
\addtolength{\jot}{0.5em}
\begin{align}
& \text{find }\quad \overline{q} \in \mathcal{Q}\quad\text{ such
that }\quad
J(\overline{q})=\min\limits_{q\in \mathcal{Q}}\,J(q), \label{PControlJ} \\
& \text{find }\quad \overline{q}_{\alpha}\in \mathcal{Q}\quad\text{
such that }\quad J_{\alpha}(\overline{q}_{\alpha})=\min\limits_{q\in
\mathcal{Q}}\,J_{\alpha}(q),  \label{PControlalfaJ}
\end{align}
\endgroup
where the cost functionals $J:\mathcal{Q}{\rightarrow}{\Bbb
R}_{0}^{+}$ and $J_{\alpha}:\mathcal{Q}{\rightarrow}{\Bbb
R}_{0}^{+}$ are given by:
\begin{small}\begin{equation}\label{defJyJalfa}
\text{i) } J(q)=\frac{1}{2}\left\| u_{q}-z_{d}\right\|
_{\mathcal{H}}^{2}+\frac{M}{2}\left\| q\right\|_{\mathcal{Q}}^{2}
,\,\, \text{ii) } J_{\alpha}(q)=\frac{1}{2}\left\| u_{\alpha q}
-z_{d}\right\| _{\mathcal{H}}^{2}+\frac{M}{2}\left\|
q\right\|_{\mathcal{Q}}^{2}
\end{equation}\end{small}
with $z_d\in \mathcal{H}$ given and $M$ a positive constant.

\par In \cite{GT1}, the authors studied boundary optimal control problems
on the heat flux  $q$ in mixed elliptic problems  and they proved
existence, uniqueness and asymptotic behavior to the optimal
solutions, when the heat transfer coefficient goes to infinity.
Similar results were obtained in \cite{GT2} for simultaneous
distributed-boundary optimal control problems on the internal energy
$g$ and the heat flux $q$ in mixed elliptic problems. In \cite{MT},
convergence results were proved for heat conduction
problems in relation to distributed optimal control problems on the
internal energy $g$ as a control variable. Parabolic control problem with Robin boundary conditions are considered in \cite{BT,BT2,CGH,GT1,MT}. Other papers on the subject are \cite{BFR,SS,SA,WY}. In this paper, our main goal is to study the existence and uniqueness of solutions and the asymptotic behaviour of the optimal control problems (\ref{PControlJ}) and (\ref{PControlalfaJ}), when $\alpha\rightarrow \infty$. Moreover, motivated by \cite{GT} we try find explicit solutions for the optimal controls and a relationship between elliptic and parabolic boundary optimal control problems, when the time goes to infinity. In this way, we consider the following family of optimization problems on the heat flux dependent of a real parameter.
\newline
For fixed $q_{0}\in \mathcal{Q}$, we define
$\mathcal{Q}_{0}=\{\lambda q_{0}:\lambda \in \mathbb{R}\}\subset \mathcal{Q}$, and we formulate the following \emph{real Neumann parabolic boundary} optimal control problems, for each $T>0$ and $\alpha >0$:
\begin{align}
& \text{find }\quad \overline{\lambda }(T)\in \mathbb{R}\quad\text{
such that }\quad
H_{T}(\overline{\lambda}(T))=\min\limits_{\lambda \in \mathbb{R}}\,H_{T}(\lambda), \label{PControlH} \\
& \text{find }\quad \overline{\lambda}_{\alpha}(T)\in
\mathbb{R}\quad\text{ such that }\quad
H_{\alpha T}(\overline{\lambda}_{\alpha}(T))=\min\limits_{\lambda \in
\mathbb{R}}\,H_{\alpha T}(\lambda ), \label{PControlalfaH}
\end{align}
where
\begin{equation}\label{defHT}
H_{T}(\lambda)=J(\lambda  q_{0}) \quad \text{and}\quad
H_{\alpha T}(\lambda)=J_{\alpha}(\lambda q_{0}).
\end{equation}
Moreover, we consider the elliptic mixed problems $P$ and $P_{\alpha}$, for each $\alpha>0$, \cite{GT1,GT2}:
\begin{equation}\label{Stationary}
-\Delta u=g\,\ \text{ in }\Omega \ \
\,\,\,\,\,\,u\big|_{\Gamma _{1}}=b\,\,\,\,\,\,\,\,-\frac{\partial
u}{\partial n}\bigg|_{\Gamma _{2}}=q
\end{equation}
\begin{equation}\label{Stationaryalpha}
-\Delta u=g\,\ \text{ in }\Omega \ \ \,\,\,\,\,\,-\frac{\partial u}{\partial n%
}\bigg|_{\Gamma _{1}}=\alpha (u-b)\,\,\,\,\,\,\,\,-\frac{\partial
u}{\partial n}\bigg|_{\Gamma _{2}}=q
\end{equation}
whose variational equalities are given by
\begin{equation}\label{PvariacionalEstacionario}
a(u,v)=L(v),\quad \forall v\in V_{0}, \,\, u\in K
\end{equation}
\begin{equation}\label{PalfavariacionalEstacionario}
a_{\alpha}(u_{\alpha },v)=L_{\alpha}(v),\quad \forall v\in V, \,\, u_{\alpha}\in V
\end{equation}
with $K=v_{0}+V_{0}$ for a given $v_{0}=b$ in $\Gamma_{1}$. For fixed $q^{*}_{0}\in Q$, we define
$Q_{0}=\{\lambda q^{*}_{0}:\lambda \in \mathbb{R}\}\subset Q$, and we formulate the following \emph{real Neumann elliptic boundary} optimal control problems, for each $\alpha >0$:
\begin{align}
& \text{find }\quad \overline{\lambda} \in \mathbb{R}\quad\text{
such that }\quad
H(\overline{\lambda})=\min\limits_{\lambda  \in \mathbb{R}}\,H(\lambda ), \label{PControlHEstacionario} \\
& \text{find }\quad \overline{\lambda}_{\alpha}\in
\mathbb{R}\quad\text{ such that }\quad
H_{\alpha}(\overline{\lambda}_{\alpha})=\min\limits_{\lambda \in
\mathbb{R}}\,H_{\alpha}(\lambda ), \label{PControlalfaHEstacionario}
\end{align}
where
\begin{equation}\label{defH}
H(\lambda)=J^{*}(\lambda q^{*}_{0}) \quad \text{and}\quad
H_{\alpha}(\lambda )=J^{*}_{\alpha}(\lambda q^{*}_{0}),
\end{equation}
with $J^{*}:Q{\rightarrow}{\Bbb
R}_{0}^{+}$ and $J^{*}_{\alpha}:Q{\rightarrow}{\Bbb
R}_{0}^{+}$ given by \cite{GT1}:
\begin{small}\begin{equation}\label{defJyJalfaEstacionario}
\text{i) } J^{*}(q)=\frac{1}{2}\left\| u_{\infty q}-z_{d}\right\|
_{H}^{2}+\frac{M}{2}\left\| q\right\|_{Q}^{2}
,\, \text{ii) } J^{*}_{\alpha}(q)=\frac{1}{2}\left\| u_{\infty\alpha q}
-z_{d}\right\| _{H}^{2}+\frac{M}{2}\left\|
q\right\|_{Q}^{2}
\end{equation}\end{small}
where $u_{\infty q}$ and $u_{\infty\alpha q}$ are the unique solutions to the variational equalities (\ref{PvariacionalEstacionario}) and (\ref{PalfavariacionalEstacionario}) respectively, $z_d\in H$ is given and $M$ is a positive constant.

The paper is structured as follows. In Section 2, we consider Neumann boundary optimal control
problems on the heat flux $q$ for heat conduction problems
(\ref{P}),(\ref{PControlJ}) and (\ref{defJyJalfa}i) and Neumann parabolic boundary
optimal control problems on the heat flux $q$ for (\ref{Palfa}),
(\ref{PControlalfaJ}) and (\ref{defJyJalfa}ii), for each $\alpha
>0$. We prove existence and uniqueness of the optimal controls and we
give the first order optimality conditions. In Section 3, for fixed
$q$, we prove asymptotic estimates and convergence results for the
system states, the adjoint states and the optimal controls, when the
heat transfer coefficient goes to infinity. In Section 4, we prove
estimates between the optimal controls of the problems
(\ref{PControlJ}) and (\ref{PControlalfaJ}) and the second component
of the simultaneous optimal controls of the problems studied in
\cite{TaBoGa}. In Section 5, for the real Neumann parabolic boundary optimal control problems (\ref{PControlH}),
(\ref{PControlalfaH}), (\ref{PControlHEstacionario}) and (\ref{PControlalfaHEstacionario}) we prove the existence and uniqueness and we find explicit solutions for the optimal control $\overline{\lambda}(t)$, $\overline{\lambda_{\alpha}}(t)$, $\overline{\lambda}$ and $\overline{\lambda_{\alpha}}$,
respectively. Moreover, monotonicity properties with respect to the data are also studied. Finally, in Section 6, convergence results of the solutions to the problems (\ref{Pvariacional}) to the solution to the problem (\ref{PvariacionalEstacionario}) are obtained, when the parameter time $t\rightarrow \infty$.

\section{Boundary Optimal Control Problems for Systems $\mathbf{S}$ and $\mathbf{S_{\alpha}}$}

Here, we prove that the functionals $J$ and $J_{\alpha}$ are
strictly convex and G\^{a}teaux differentiable in $\mathcal{Q}$.
Moreover, we obtain the existence and uniqueness of the boundary
optimal controls
 $\overline{q}$ and  $\overline{q}_{\alpha}$ and we give the optimality
 conditions en terms of the adjoint states, for the optimal control problems (\ref{PControlJ}) and (\ref{PControlalfaJ}), respectively.

Following \cite{Li,MT, Tro}, we
 define the application $C:\mathcal{Q}\rightarrow L^2 (V_0)$ such that
$C(q)=u_q-u_0$, where $u_0$ is the solution of problem
(\ref{Pvariacional}) for $q=0$.

We consider $\Pi:\mathcal{Q}\times \mathcal{Q}\rightarrow
\mathbb{R}$ and $\mathcal{L}:\mathcal{Q}\rightarrow \mathbb{R}$
defined by the expressions
\[\Pi(q,\eta)=(C(q),C(\eta))_\mathcal{H}+M(q,\eta)_\mathcal{Q}\quad\forall q,\eta\in
\mathcal{Q}
\]
\[
\mathcal{L}(q)=(C(q),z_d-u_0)_\mathcal{H} \quad\forall q\in
\mathcal{Q}
\]
and we prove the following result
\begin{lem}\label{lema1}
\begin{enumerate}
                 \item [i)] $C$ is a linear and continuous functional.
                 \item [ii)] $\Pi$ is a bilinear, symmetric, continuous
                 form and coercive in $\mathcal{Q}$.
                 \item [iii)] $\mathcal{L}$ is linear and continuous functional in $\mathcal{Q}$.
                 \item [iv)] $J$ can be write as:
\[
                 J(q)=\frac{1}{2}\Pi(q,q)-\mathcal{L}(q)+\frac{1}{2}||u_0-z_d||^2_{\mathcal{H}}, \quad  \forall q\in
                 \mathcal{Q}.
                 \]
    \item [v)] $J$ is a strictly convex functional on
                 $\mathcal{Q}$, that is, \, $\forall q_1,q_2\in
                    \mathcal{Q},\,\forall t\in[0,1]$
                 \begin{equation*}\begin{split}
                    &(1-t)J(q_2)+tJ(q_1)-J((1-t)q_2+tq_1)\geq
                    \frac{Mt(1-t)}{2}||q_2-q_1||^2_{\mathcal{Q}}.
                 \end{split}\end{equation*}

                 \item [vi)] There exists a unique optimal control $\overline{q}\in
                 \mathcal{Q}$ such that
\[
                 J(\overline{q})=\min\limits_{q\in \mathcal{Q}}J(q).
                 \]
               \end{enumerate}
\end{lem}

\begin{proof}

It follows from \cite{Li,MT} and
\begin{small}\[ (1-t)J(q_2)+tJ(q_1)-J((1-t)q_2+tq_1)=
                    \frac{t(1-t)}{2}\left[||u_{q_2}-u_{q_1}||^2_{\mathcal{H}}+M||q_2-q_1||^2_{\mathcal{Q}}\right],
\]\end{small}
$\forall q_1,q_2\in
                    \mathcal{Q},\,\,\forall t\in[0,1]$.
\end{proof}

Now, we define the adjoint state $p_q$ corresponding to the system (\ref{P})
for each $q\in \mathcal{Q}$, as the unique solution of the following
mixed parabolic problem:
\begin{equation*}
- \frac{\partial p_q}{\partial t}-\Delta p_q=u_q-z_d\,\ \text{in
}\Omega, \ \ \,\,\,\,\,\,p_q\big|_{\Gamma
_{1}}=0,\,\,\,\,\,\,\,\,\frac{\partial p_q}{\partial n}\bigg|_{\Gamma
_{2}}=0,\,\,\,\,\,\,\,\,p_q(T)=0,
\end{equation*}
whose variational formulation is given by
\begin{equation}\label{pvariacional}
\left\{
\begin{array}{l l}
p_q \in L^2(V_0), \,\, p_q(T)=0\quad\text{and } \dot{p}_q\in
L^2(V_0') \text{ such that }\\ -\langle \dot{p}_q(t),
v\rangle+a(p_q(t),v)=(u_q(t)-z_d(t),v)_H, \quad \forall v\in V_0,
\end{array}
\right.
\end{equation}
and we consider the following properties of the functional $J$, which follows \cite{Li,MT,TaBoGa}.
\begin{lem}\label{lema2}
\begin{enumerate}\item [i)] The adjoint state $p_q$ satisfies:
\[
(C(\eta),u_q-z_d)_\mathcal{H}=-(\eta,p_q)_\mathcal{Q},\quad \forall
q,\eta\in \mathcal{Q}.
\]
                 \item [ii)] The functional $J$ is
                 G\^{a}teaux differentiable and $J^{\prime}$ is given
                 by:
                   \begin{equation*}\begin{split}
                  \langle J^{\prime}(q),\eta-q\rangle
                  &=(u_\eta-u_q,u_q-z_d)_\mathcal{H}+M(q,\eta-q)_\mathcal{Q}\\&
                  = \Pi(q,\eta-q)-\mathcal{L}(\eta-q), \quad \forall q,\eta\in \mathcal{Q}.
                   \end{split}\end{equation*}
                 \item [iii)] The G\^{a}teaux derivative of $J$ can be write
                 as:
\[
                 J^{\prime}(q)=Mq-p_q\quad \forall q\in \mathcal{Q}.
                 \]
                 \item [iv)] The optimality condition for the optimal control problem (\ref{PControlJ}) is given by
\[
                 M \overline{q}-p_{\overline{q}}=0 \quad \text{in } \mathcal{Q}.
                 \]
                \end{enumerate}
\end{lem}

Next, we define the application $C_{\alpha}:\mathcal{Q}\rightarrow
L^2 (V)$ such that $C_{\alpha}(q)=u_{\alpha q}-u_{\alpha 0}$, where
$u_{\alpha 0}$ is the solution of the variational problem
(\ref{Palfavariacional}) for $q=0$.

If we consider $\Pi_{\alpha}:\mathcal{Q}\times
\mathcal{Q}\rightarrow \mathbb{R}$ and
$\mathcal{L}_{\alpha}:\mathcal{Q}\rightarrow \mathbb{R}$ defined by
\[
\Pi_{\alpha}(q,\eta)=(C_{\alpha}(q),C_{\alpha}(\eta))_\mathcal{H}+
M(q,\eta)_\mathcal{Q}\quad\forall q,\eta\in \mathcal{Q}
\]
\[
\mathcal{L}_{\alpha}(q)=(C_{\alpha}(q),z_d-u_{\alpha 0})_\mathcal{H}
\quad\forall q\in \mathcal{Q}
\]
in similar way to Lemma \ref{lema1} and \ref{lema1} we have the following result.
\begin{lem}
(i) There exists a unique optimal control $\overline{q}_{\alpha}\in
                 \mathcal{Q}$ such that
\[ J_{\alpha}(\overline{q}_{\alpha})=\min\limits_{q\in \mathcal{Q}}J_{\alpha}(q).
                 \]
(ii) The Gateaux derivate of $J_{\alpha}$ can be write as:
\begin{equation}
J_{\alpha}^{\prime}(q)=Mq-p_{\alpha q}\quad \forall q\in \mathcal{Q}.
\end{equation}
and the optimality condition for the optimal control problem (\ref{PControlalfaJ}) is given by:
\begin{equation}
M \overline{q}_{\alpha}-p_{\alpha\overline{q}_{\alpha}}=0 \quad \text{in } \mathcal{Q}
\end{equation}
where the adjoint state $p_{\alpha q}$ corresponding to
(\ref{Palfa}) for each  $q\in \mathcal{Q},$ as the unique solution
of
\begin{equation}
- \frac{\partial p_{\alpha q}}{\partial t}-\Delta p_{\alpha
q}=u_{{\alpha}q}-z_d\,\ \text{in }\Omega, \,\, -\frac{\partial
p_{\alpha q}}{\partial n}\bigg|_{\Gamma _{1}}=\alpha p_{\alpha q},
\,\,\frac{\partial p_{\alpha q}}{\partial n}\bigg|_{\Gamma
_{2}}=0,\, p_{\alpha q}(T)=0\nonumber
\end{equation}
whose variational formulation is given by
\begin{equation}\label{palphavariacional1}
\left\{
\begin{array}{l l}
p_{\alpha q} \in L^2(V), \,\, p_{\alpha q}(T)=0\quad\text{and }\,
\dot{p}_{\alpha q}\in L^2(V') \text{ such that } \\ -\langle
\dot{p}_{\alpha q}(t), v\rangle+a_{\alpha}(p_{\alpha
q}(t),v)=(u_{\alpha q}(t)-z_d(t),v)_H, \, \forall v\in V
\end{array}
\right.
\end{equation}
for each $\alpha >0$.
\end{lem}

\section{Convergence  of  Neumann  Boundary  Optimal \linebreak Control Problems when $\mathbf{\alpha\rightarrow\infty}$}

Now, for fixed $q\in Q$, we obtain estimates on $u_{\alpha q}$ and
$p_{\alpha q}$ uniformly when $\alpha > 1$. Next, we prove strong
convergence for $q_{\alpha}$, $u_{\alpha q}$ and $p_{\alpha q}$, when $\alpha$ goes to infinity.

\begin{prop} \label{asintotico1} (i) If $u_{q}$ and $u_{\alpha q}$ are
the unique solutions to the variational equalities
(\ref{Pvariacional}) and (\ref{Palfavariacional}) respectively, we
have the estimation

\begin{small}\begin{equation}\label{estim1}
   ||\dot{u}_{\alpha q}||_{L^{2}(V'_{0})}+ ||u_{\alpha q}||_{L^{\infty}(H)}+
    ||u_{\alpha q}||_{L^{2}(V)}+\sqrt{(\alpha-1)}
    ||u_{\alpha q}-b||_{L^{\infty}(L^2(\Gamma_1))}\leq C
\end{equation}\end{small}
for all $\alpha>1$, where the constant $C$ depend only on the norms
$||\dot{u}_{q}||_{L^{2}(V'_{0})}$, $||\dot{u}_{q}||_{L^{2}(V')}$,
$||\nabla u_{q}||_{\mathcal{H}}$, $||u_{q}||_{L^{2}(V)}$,
$||u_{q}||_{L^{\infty}(H)}$, $||g||_{\mathcal{H}}$,
$||q||_{\mathcal{Q}}$ and the coerciveness constant
$\lambda_1$ of the bilinear form $a_{1}$.

(ii) For fixed $q\in \mathcal{Q}$ we have
$u_{\alpha q}\rightarrow u_q$ strongly in $L^2(V)\cap L^{\infty}(H)$
and $\dot{u}_{\alpha q}\rightarrow \dot{u}_q$ strongly in
$L^2(V_0')$, when $\alpha\rightarrow\infty$.
\end{prop}

\begin{proof}
Taking $v=u_{\alpha q}(t)-u_q(t)\in V$ in the variational equation
(\ref{Palfavariacional}), taking into account that $u_q(t)\big|_{\Gamma_1}=b$, by using Young's inequality and integrating between $[0,T]$, we obtain
\begin{equation*}\begin{split}
&\quad \frac{1}{2}||u_{\alpha
q}(T)-u_q(T)||_H^2+\frac{\lambda_1}{2}||u_{\alpha
q}-u_q||^2_{L^2(V)}+(\alpha-1)||u_{\alpha
q}-b||^2_{L^2(L^2(\Gamma_1))}
\\& \leq
\frac{2}{\lambda_1}[||g||_{\mathcal{H}}^2+\|\gamma_{0}\|^{2}||q||^2_{\mathcal{Q}}
+||\nabla u_{q}||^2_{\mathcal{H}}+|| \dot{u}_{q}||^2_{L^2(V')}],
\end{split}\end{equation*}
where $\gamma_{0}$ is the trace operator on $\Gamma$.

Here, we prove that there exists a positive constant $K$ independent of $\alpha$ and it depends of
\begin{small}\[
K=K(\lambda_{1},||u_q||_{L^{\infty}(H)},||u_q||_{L^2(V)},||g||_{\mathcal{H}},||q||_{\mathcal{Q}},||\nabla
u_{q}||_{\mathcal{H}},|| \dot{u}_{q}||_{L^2(V')})
\]\end{small}
such that for all $\alpha> 1$, we have:
\begin{equation}\label{K}
\quad||u_{\alpha q}||_{L^{\infty}(H)}+||u_{\alpha
q}||_{L^2(V)}+\sqrt{(\alpha-1)}||u_{\alpha
q}-b||_{L^2(L^2(\Gamma_1))}\leq K.
\end{equation}
Next, taking $v\in V_{0}$ in the variational equality
(\ref{Palfavariacional}) and subtracting the variational equality (\ref{Pvariacional}), we
have
\begin{equation*}
( \dot{u}_{\alpha q}(t)-\dot{u}_{q}(t), v)_{H} \leq
||u_{q}(t)-u_{\alpha q}(t)||_{V}||v||_{V_0}\quad\forall v\in V_0,
\end{equation*}
and integrating in $[0,T]$, we obtain
$||\dot{u}_{\alpha q}-\dot{u}_{q}||_{L^2(V_0')} \leq
||u_{q}-u_{\alpha q}||_{L^2(V)}$. Next, by using (\ref{K}), we have that there exists a positive
constant $C=C(K,|| \dot{u}_{q}||_{L^2(V'_{0})})$
such that (\ref{estim1}) holds.

(ii) Let fixed $q\in \mathcal{Q}$ be, we consider a sequence $\{u_{\alpha_n
q}\} $ in $L^2(V)\cap L^{\infty}(H)$ and by estimation
(\ref{estim1}), we have that $||u_{\alpha_{n} q}||_{L^2(V)}\leq C$ and $||\dot{u}_{\alpha_n q}||_{L^2(V_0')}\leq C$, therefore, there exists a subsequence $\{u_{\alpha_{n}q}\}$ which is
weakly convergent to $w_q\in L^2(V)$ and weakly* in $L^{\infty}(H)$
and there exists a subsequence $\{\dot{u}_{\alpha_{n}q}\}$ which is
weakly convergent to $\dot{w}_q\in L^2(V_0')$. Now, from  the third term of left hand side of (\ref{K}) and the weak
lower semicontinuity of the norm in $L^2(L^2(\Gamma_{1}))$, we have that $w_q=b$ on $\Gamma_1$
and therefore $w_{q}-v_{b}\in L^2(V_{0})$. Next, we prove that
$w_{q}$ satisfies $
\langle \dot{w}_{q}(t), v\rangle+a(w_{q}(t),v)=L(t,v)$, $\forall v\in V_0$ and $w_{q}(0)=v_b$ with $\dot{w}_{q}\in L^2(V_0')$. Therefore, by
uniqueness of the solution of the variational problem
(\ref{Pvariacional}), we obtain $w_{q}=u_{q}$. That is, when
$\alpha\rightarrow\infty$ we have
$$u_{\alpha q}\rightharpoonup u_q\,\,\text{in}\,\,L^{2}(V),\quad u_{\alpha q}\stackrel{*}{\rightharpoonup} u_q\,\,
\text{in}\,\,L^{\infty}(H)\quad\text{and}\quad \dot{u}_{\alpha
q}\rightharpoonup \dot{u}_q\,\,\text{in}\,\,L^{2}(V_0').$$
Now, we have
\begin{small}\begin{equation*}\begin{split}
&\quad\frac{1}{2}||u_{\alpha q}(T)-u_q(T)||_H^2+\lambda_1||u_{\alpha
q}-u_q||^2_{L^2(V)}+(\alpha-1)||u_{\alpha
q}-u_q||^2_{L^2(L^2(\Gamma_1))}
\\&\leq \int\limits_0^T \{
L(t,u_{\alpha q}(t)-u_q(t))-a(u_{q}(t),u_{\alpha
q}(t)-u_q(t))-\langle \dot{u}_{q}(t),u_{\alpha q}(t)-u_q(t)\rangle \}
dt
\end{split}\end{equation*}\end{small}
and by using the weak convergence of $u_{\alpha q}$ to $u_q$, we
prove the strong convergence in $L^2(V)$. Next, taking into account
that
\begin{small}\begin{equation*}\begin{split}
\quad||u_{\alpha q}-u_q||^2_{L^2(L^2(\Gamma_1))} &\leq
\frac{1}{\alpha-1}\int\limits_0^T \{L(t,u_{\alpha
q}(t)-u_q(t))-a(u_{q}(t),u_{\alpha q}(t)-u_q(t))\\&-\langle
\dot{u}_{q}(t),u_{\alpha q}(t)-u_q(t)\rangle \} dt
\end{split}\end{equation*}\end{small}
and the weak convergence of $u_{\alpha q}$ to $u_q$, we prove the
strong convergence in $L^2(L^2(\Gamma_1))$. Now, from the
variational equalities (\ref{Pvariacional}) and (\ref{Palfavariacional}), we have
\begin{equation*}
||\dot{u}_{\alpha q}-\dot{u}_{q}||^2_{L^2(V_0')} \leq
||u_{q}-u_{\alpha q}||^2_{L^2(V)}\rightarrow
0,\quad \text{when}\quad \alpha \rightarrow
\infty.
\end{equation*}
We deduce that $\dot{u}_{\alpha q}$ is strong convergent to
$\dot{u}_{q}$ en $L^2(V_0')$. Finally, we have
\begin{small}\begin{equation*}\begin{split}||u_{\alpha q}-u_q||^2_{L^{\infty}(H)}&\leq 2(||g||_{\mathcal{H}}
+\|\gamma_{0}\|||q||_{\mathcal{Q}} +|| u_{q}||_{L^2(V_0)}+
||\dot{u}_{q}||_{\mathcal{H}}) ||u_{\alpha
q}-u_q||_{L^2(V)}\end{split}\end{equation*} \end{small}
and from the strong
convergence of $u_{\alpha q}$ to $u_q$ in $L^2(V)$, we prove that
$u_{\alpha q}$ is strongly convergent to $u_q$ in $L^\infty(H)$, when $\alpha \rightarrow \infty$.
\end{proof}

\begin{prop}\label{asintotico2} (i) If $p_{q}$ and $p_{\alpha q}$ are
the unique solutions to the variational equalities
(\ref{pvariacional}) and (\ref{palphavariacional1}) respectively, we
have the estimation
\begin{equation}\label{estimacion2}
    ||\dot{p}_{\alpha q}||_{L^{2}(V'_{0})}+||p_{\alpha q}||_{L^{\infty}(H)}+
    ||p_{\alpha q}||_{L^{2}(V)}+\sqrt{(\alpha-1)}
    ||p_{\alpha q}||_{L^{2}(L^2(\Gamma_1))}\leq C
\end{equation}
for all $\alpha>1$, where the constant $C$ depend of the norms
$||\dot{p}_{q}||_{L^{2}(V'_{0})}$,\linebreak $||\dot{p}_{q}||_{L^{2}(V')}$,
$||\nabla p_{q}||_{\mathcal{H}}$, $||p_{q}||_{L^{2}(V)}$,
$||p_{q}||_{L^{\infty}(H)}$, $||g||_{\mathcal{H}}$,
$||q||_{\mathcal{Q}}$, $||z_d||_{\mathcal{H}}$, \linebreak
$||\dot{u}_{q}||_{L^{2}(V')}$, $||\nabla u_{q}||_{\mathcal{H}}$,
$||u_{q}||_{L^{2}(V)}$, $||u_{q}||_{L^{\infty}(H)}$ and of the
coerciveness constant $\lambda_1$.

(ii) For fixed $q\in \mathcal{Q}$, we have that
$p_{\alpha q}\rightarrow p_q$ strongly in  $L^2(V)\cap
L^{\infty}(H)$ and $\dot{p}_{\alpha q}\rightarrow \dot{p}_q$
strongly in $L^2(V_0')$, when $\alpha\rightarrow\infty$.
\end{prop}

\begin{proof}
Let fixed $q\in \mathcal{Q}$ be, the estimation (\ref{estimacion2}) follows with an analogous reasoning to
Proposition \ref{asintotico1}. We have that there exists a subsequence $\{p_{\alpha_{n}q}\}$ which
is weakly convergent to $\eta_q\in L^2(V)$ and weakly* in
$L^{\infty}(H)$. From the weak semicontinuity of the norm, we have
that $\eta_q=0$ on $\Gamma_1$ and therefore $\eta_q\in L^2(V_{0})$.
Moreover, $\eta_q$ satisfies
\begin{equation*}
-\langle \dot{\eta}_{q}(t), v\rangle
+a(\eta_q(t),v)=(u_{q}(t)-z_d(t),v)_H \quad \forall v\in V_{0}
\end{equation*}
and $\eta_q(T)=0$ with $\dot{\eta}_{q}\in L^2(V_0')$. Therefore, by
uniqueness of the solution of the variational problem
(\ref{pvariacional}), we obtain $\eta_{q}=p_{q}$ and when
$\alpha\rightarrow\infty$ we have
$$p_{\alpha q}\rightharpoonup p_q\,\,\text{in}\,\,L^{2}(V),\quad p_{\alpha q}\stackrel{*}{\rightharpoonup}
p_q\,\,\text{in}\,\,L^{\infty}(H)\quad \text{and} \quad
\dot{p}_{\alpha q}{\rightharpoonup}
\dot{p}_q\,\,\text{in}\,\,L^{2}(V_0').$$ Finally, the strong
convergence of $p_{\alpha q}$ to  $p_q$ in  $L^2(V)\cap
L^{\infty}(H)$ and of $\dot{p}_{\alpha q}$ to $\dot{p}_q$ in norm
$L^2(V_0')$ is obtained in a similar way that in Proposition \ref{asintotico1}.
\end{proof}

Now, we consider the boundary optimal control problems
(\ref{PControlJ}) and (\ref{PControlalfaJ}) and our goal is to prove
the following theorem:
\begin{thm} Let $\overline{q}$ and $\overline{q}_{\alpha}$ the unique solutions of the optimal control problems (\ref{PControlJ}) and (\ref{PControlalfaJ}), respectively. Then, we have that
$\overline{q}_{\alpha}\rightarrow \overline{q}$ strongly in
$\mathcal{Q}$, when the parameter $\alpha\rightarrow \infty$. Moreover, the system state and the adjoint state
satisfy
$(u_{\alpha\overline{q}_{\alpha}},\dot{u}_{\alpha\overline{q}_{\alpha}})
\rightarrow (u_{\overline{q}},\dot{u}_{\overline{q}})$ and
$(p_{\alpha\overline{q}_{\alpha}},\dot{p}_{\alpha\overline{q}_{\alpha}})\rightarrow
(p_{\overline{q}},\dot{p}_{\overline{q}})$ strongly in $L^2(V)\times
L^2(V_0')$.
\end{thm}
\begin{proof}
We will do the proof in three steps.

\textbf{Step 1.} From the estimation (\ref{estim1}) for $q=0$, there
exists a constant $C_{1}>0$ such that $||u_{\alpha 0}||_{\mathcal{H}}\leq ||u_{\alpha 0}||_{L^2(V)}\leq C_{1}, \,\, \forall \alpha>1$, and from
$J_{\alpha}(\overline{q}_{\alpha})\leq J_{\alpha}(0)$, we have
\begin{equation}\frac{1}{2}||u_{\alpha\overline{q}_{\alpha}}-z_d||^2_{\mathcal{H}}+\frac{M}{2}||\overline{q}_{\alpha}||^2_{\mathcal{Q}}\leq \frac{1}{2}
||u_{\alpha 0}-z_d||^2_{\mathcal{H}}.\nonumber\end{equation}
Therefore, there exist positive constants $C_{2}$ and $C_{3}$ such
that
$$||u_{\alpha\overline{q}_{\alpha}}||_{\mathcal{H}}\leq C_2\quad \text{and}\quad ||\overline{q}_{\alpha}||_{\mathcal{Q}}\leq C_{3},\quad \forall \alpha>1.$$
Now, by an analogous reasoning to the estimates (\ref{estim1}) and (\ref{estimacion2}), there exist positive
constants $C_{4}$  and $C_{5}$ such that, for all $\alpha > 1$, we obtain
\begin{equation}
||u_{\alpha\overline{q}_{\alpha}}||_{L^{2}(V)}+||\dot{u}_{\alpha\overline{q}_{\alpha}}||_{L^2(V_0')}+\sqrt{(\alpha-1)}
    ||u_{\alpha\overline{q}_{\alpha}}-b||_{L^{2}(L^2(\Gamma_1))}\leq C_4\nonumber
\end{equation}
\begin{equation}
||p_{\alpha\overline{q}_{\alpha}}||_{L^{2}(V)}+||\dot{p}_{\alpha\overline{q}_{\alpha}}||_{L^2(V_0')}+\sqrt{(\alpha-1)}
    ||p_{\alpha\overline{q}_{\alpha}}||_{L^{2}(L^2(\Gamma_1))}\leq
    C_5.\nonumber
\end{equation}
From the previous estimations, we have that there exist $f\in
\mathcal{Q}$, $\mu\in L^2(V)$, $\dot{\mu}\in L^2(V_0')$,
$\rho\in L^2(V)$ and $\dot{\rho}\in L^2(V_0')$ such that
$$\overline{q}_{\alpha}\rightharpoonup f\in \mathcal{Q}, \quad  u_{\alpha\overline{q}_{\alpha}}\rightharpoonup \mu\in L^2(V),\quad
\dot{u}_{\alpha\overline{q}_{\alpha}}\rightharpoonup \dot{\mu}\in
L^2(V'_{0}),$$
$$p_{\alpha\overline{q}_{\alpha}}\rightharpoonup \rho\in L^2(V),\quad
\dot{p}_{\alpha\overline{q}_{\alpha}}\rightharpoonup \dot{\rho}\in
L^2(V'_{0}).$$

\textbf{Step 2.}  Taking into account the weak convergence of
$u_{\alpha\overline{q}_{\alpha}}$ to $\mu$ in $L^2(V)$ and the
estimation $\sqrt{(\alpha-1)}
    ||u_{\alpha\overline{q}_{\alpha}}-b||_{L^{2}(L^2(\Gamma_1))}\leq C_{4}$, in similar way that Proposition \ref{asintotico1}, we obtain that $\mu=u_{f}$.
Moreover, for the adjoint state, we have that
$p_{\alpha\overline{q}_{\alpha}}$ is weakly convergent to $\rho$ in
$L^2(V)$ and from estimation $
\sqrt{(\alpha-1)}
    ||p_{\alpha\overline{q}_{\alpha}}||_{L^{2}(L^2(\Gamma_1))}\leq C_5$, in similar way that Proposition \ref{asintotico2}, we obtain that $\rho=p_{f}$.
Therefore, we have $u_{\alpha\overline{q}_{\alpha}}\rightharpoonup
u_{f}$ in $L^2(V)$ and $p_{\alpha\overline{q}_{\alpha}}\rightharpoonup p_{f}$ in $L^2(V)$. Now, the optimality condition for the optimal control problem (\ref{PControlalfaJ}) is given by
$(M
\overline{q}_{\alpha}-p_{\alpha\overline{q}_{\alpha}},\eta)_{\mathcal{Q}}=0$, $\forall \eta\in \mathcal{Q}$, and taking into account that
\[
p_{\alpha\overline{q}_{\alpha}}\rightharpoonup p_{f}\quad\text{
in}\quad L^2(V)\quad \text{and}\quad
\overline{q}_{\alpha}\rightharpoonup f\quad\text{in}\quad
\mathcal{Q},\] we obtain $-p_{f}+Mf=0$
and by uniqueness of the optimal control we deduce that
$f=\overline{q}$. Therefore $u_{f}=u_{\overline{q}}$, $p_{f}=p_{\overline{q}}$, $\dot{u}_{f}=\dot{u}_{\overline{q}}$ and $\dot{p}_{f}=\dot{u}_{\overline{q}}$.

\textbf{Step 3.} We have, for all $q\in \mathcal{Q}$
\[
J(\overline{q})=\frac{1}{2}||u_{\overline{q}}-z_d||_{\mathcal{H}}^2+\frac{M}{2}||\overline{q}||^2_{\mathcal{Q}}
\leq\liminf\limits_{\alpha\rightarrow\infty}\left[\frac{1}{2}||u_{\alpha\overline{q}_{\alpha}}-z_d||_{\mathcal{H}}^2
+\frac{M}{2}||\overline{q}_{\alpha}||^2_{\mathcal{Q}}\right]\leq
\]
\[
\limsup\limits_{\alpha\rightarrow\infty}
\left[\frac{1}{2}||u_{\alpha\overline{q}_{\alpha}}-z_d||_{\mathcal{H}}^2+\frac{M}{2}||\overline{q}_{\alpha}||^2_{\mathcal{Q}}\right]
\leq\limsup\limits_{\alpha\rightarrow\infty} J_{\alpha}(q)=
\]
\[
\lim\limits_{\alpha\rightarrow\infty}\left[\frac{1}{2}||u_{\alpha
q}-z_d||_{\mathcal{H}}^2+\frac{M}{2}||q||^2_{\mathcal{Q}}\right]=\frac{1}{2}||u_{q}-z_d||_{\mathcal{H}}^2+\frac{M}{2}||q||^2_{\mathcal{Q}}=J(q).
\]
By taking infimum on $q$, all the above inequalities become
equalities and therefore
\[
\lim\limits_{\alpha\rightarrow\infty}\left[||u_{\alpha\overline{q}_{\alpha}}-z_d||_{\mathcal{H}}^2+M||\overline{q}_{\alpha}||^2_{\mathcal{Q}}\right]
=||u_{\overline{q}}-z_d||_{\mathcal{H}}^2+M||\overline{q}||^2_{\mathcal{Q}},
\]
that is
\begin{equation*}
\lim\limits_{\alpha\rightarrow\infty}||(\sqrt{M}
\overline{q}_{\alpha},u_{\alpha\overline{q}_{\alpha}}-z_d)||_{\mathcal{Q}\times
\mathcal{H}}^2=||(\sqrt{M}\overline{q},u_{\overline{q}}-z_d)||^2_{\mathcal{Q}\times
\mathcal{H}}.
\end{equation*}
The previous equality, the convergence
$\overline{q}_{\alpha}\rightharpoonup \overline{q}$ in $\mathcal{Q}$
and $u_{\alpha\overline{q}_{\alpha}}\rightharpoonup
u_{\overline{q}}$ in $L^2(V)$ imply that
$(\overline{q}_{\alpha},u_{\alpha\overline{q}_{\alpha}})\rightarrow(\overline{q},u_{\overline{q}})$
strongly in $\mathcal{Q}\times\mathcal{H}$, when
$\alpha\rightarrow\infty$.
\newline
Finally, if we take
$v=u_{\alpha\overline{q}_{\alpha}}(t)-u_{\overline{q}}(t)\in V$ in
(\ref{Palfavariacional}) for $u=u_{\alpha\overline{q}_{\alpha}}$ and if we call
$z_{\alpha}=u_{\alpha\overline{q}_{\alpha}}-u_{\overline{q}}$, we have
\begin{equation*}\begin{split}
&\quad\lambda_1||z_{\alpha}(t)||_{V}^2 \leq
(g(t)-\dot{u}_{\overline{q}}(t),z_{\alpha}(t))_H-(\overline{q}_{\alpha}(t),z_{\alpha}(t))_Q-a(u_{\overline{q}}(t),z_{\alpha}(t)).
\end{split}\end{equation*}
Integrating between 0 and $T$ and taking into account that $z_{\alpha}\rightharpoonup 0$ weakly in $L^2(V)$, $z_{\alpha}$ is bounded independent of $\alpha$,  and
$\overline{q}_{\alpha}\rightarrow \overline{q}$ strongly in
$\mathcal{Q}$ when $\alpha\rightarrow \infty$, following \cite{BT2}, we obtain
$$\int\limits_0^T\left[
(g(t)-\dot{u}_{\overline{q}}(t),z_{\alpha}(t))_H-(\overline{q}_{\alpha}(t),z_{\alpha}(t))_Q-a(u_{\overline{q}}(t),z_{\alpha}(t))\right]dt \rightarrow 0.$$

Next, we have $\lim\limits_{\alpha\rightarrow\infty}||z_{\alpha}||_{L^2(V)}=0$. From the variational equalities (\ref{Pvariacional}) and (\ref{Palfavariacional}), we have
\begin{equation*}(\dot{z}_{\alpha}(t),
v)_{H}+a(z_{\alpha}(t),v)=(\overline{q}(t)-\overline{q}_{\alpha}(t),v)_Q,
\quad \forall v\in V_0.
\end{equation*}
Therefore $|| \dot{z}_{\alpha}(t)||^2_{V_0'}\leq2||
z_{\alpha}(t)||^2_{V}+2||\gamma_{0}||^2||\overline{q}(t)-\overline{q}_{\alpha}(t)||^2_Q$ and integrating on $[0,T]$, we obtain
\begin{equation*}|| \dot{z}_{\alpha}||^2_{L^2(V_0')}\leq2|| z_{\alpha}||^2_{L^2(V)}+2||\gamma_{0}||^2||\overline{q}-\overline{q}_{\alpha}||^2_{\mathcal{Q}}.
\end{equation*}
Since that $\overline{q}_{\alpha}\rightarrow \overline{q}$
strongly in $\mathcal{Q}$ and
$u_{\alpha\overline{q}_{\alpha}}\rightarrow u_{\overline{q}}$
strongly in $L^2(V)$ when $\alpha\rightarrow\infty$,
we can say that $\dot{z}_{\alpha}\rightarrow 0$ strongly in
$L^2(V_0')$, that is
$\dot{u}_{\alpha\overline{q}_{\alpha}}\rightarrow
\dot{u}_{\overline{q}}$  strongly in $L^2(V_0')$. In similar way, we prove that $(p_{\alpha\overline{q}_{\alpha}},\dot{p}_{\alpha\overline{q}_{\alpha}})\rightarrow (p_{\overline{q}},\dot{p}_{\overline{q}})$ strongly in $L^2(V)\times L^2(V_0')$, when $\alpha\rightarrow \infty$.
\end{proof}

\section{Estimations  between the Optimal Controls}

In this Section, we obtain estimations between the solutions of the some Neumann boundary optimal control problems and the solutions of the simultaneous distributed-boundary optimal control problems studied in \cite{TaBoGa}.

\subsection{Estimations with respect to the problem $\mathbf{S}$}
We consider the Neumann boundary optimal control problem
\begin{equation}\label{42}
\text{find}\quad \overline{q}\in\mathcal{Q} \quad\text{such that}\quad J_1(\overline{q})=\min\limits_{q\in\mathcal{Q}}J_1(q)\qquad \text{for fixed}\quad g\in\mathcal{H},
\end{equation}
where $J_1$ is the cost functional given in (\ref{PControlJ}) plus the constant $\frac{M_1}{2}||g||^2_{\mathcal{H}}$, that is,
$J_1:\mathcal{Q}\rightarrow \mathbb{R}^+_0$ is given by
\begin{equation*}
J_1(q)=\frac{1}{2}||u_q-z_d||^2_{\mathcal{H}}+\frac{M_1}{2}||g||^2_{\mathcal{H}} +\frac{M}{2}||q||^2_{\mathcal{Q}}     \qquad (\text{fixed} \quad g\in \mathcal{H}),
\end{equation*}
where $u_q$ is the unique solution of the problem (\ref{Pvariacional}) for fixed $g$.
\begin{rem} The functional $J^{+}(g,q)$ defined in \cite[see (7)]{TaBoGa} and the functional $J_1$ previously defined, satisfy the following elemental estimation
\begin{equation*}
J^{+}(\overline{\overline{g}},\overline{\overline{q}})\leq J_1(\overline{q}),\quad \forall g\in\mathcal{H}.
\end{equation*}\end{rem}
In the following theorem we obtain estimations between the solution of the boundary optimal control problem (\ref{42}) and the second component of the solution of to simultaneous distributed-boundary optimal control problem studied in \cite{TaBoGa}.
\begin{thm}
If $(\overline{\overline{g}},\overline{\overline{q}})\in\mathcal{H}\times\mathcal{Q}$ is the unique solution to the distributed-boundary optimal control problem in \cite[see (7)]{TaBoGa} and  $\overline{q}$ is the unique solution to the optimal control problem (\ref{42}), then
\begin{equation}\label{440}
||\overline{q}-\overline{\overline{q}}||_{\mathcal{Q}}\leq\frac{||\gamma_0||}{\lambda_0 M}||u_{\overline{\overline{g}}\,\overline{\overline{q}}}-u_{g\,\overline{q}}||_{\mathcal{H}},\quad \forall g \in \mathcal{H}
\end{equation}
where $\gamma_0$ is the trace operator with $||\gamma_{0}||=\sup\limits_ {v\in V-\{0\}}\frac{||\gamma_{0}(v)||_{Q}}{||v||_{V}}$ and $\lambda_0$ the coerciveness constant of the bilinear form $a$.
\end{thm}
\begin{proof}
By the optimality condition for $\overline{q}$, for fixed $g\in \mathcal{H}$, we have $
(M\overline{q}-p_{g\,\overline{q}},\eta)_{\mathcal{Q}}=0$, $\forall\eta\in\mathcal{Q}$, and taking $\eta=\overline{\overline{q}}-\overline{q}$ we obtain \begin{equation}\label{47}
(M\overline{q}-p_{g\,\overline{q}},\overline{\overline{q}}-\overline{q})_{\mathcal{Q}}=0.
\end{equation}
On the other hand, if we take $h=0\in\mathcal{H}$ in the optimality condition for $(\overline{\overline{g}}, \overline{\overline{q}})$ given in \cite{TaBoGa}, we have $(M\overline{\overline{q}}-p_{\overline{\overline{g}}\,\overline{\overline{q}}},\eta)_{\mathcal{Q}}=0$, $\forall\eta\in\mathcal{Q}$, next, taking  $\eta=\overline{q}-\overline{\overline{q}}$, we obtain
\begin{equation}\label{48}
(-M\overline{\overline{q}}+p_{\overline{\overline{g}}\,\overline{\overline{q}}},\overline{\overline{q}}-\overline{q})_{\mathcal{Q}}=0.
\end{equation}
By adding the expressions (\ref{47}) and (\ref{48}), we have
$$\left(M(\overline{q}-\overline{\overline{q}})+(p_{\overline{\overline{g}}\,\overline{\overline{q}}}-p_{g\,\overline{q}}),\overline{\overline{q}}-\overline{q}\right)_{\mathcal{Q}}=0.$$
Here, by the Cauchy-Schwarz inequality and the trace theorem, we have
\begin{equation*}
||\overline{\overline{q}}-\overline{q}||_{\mathcal{Q}} \leq\frac{||\gamma_0||}{M}||p_{\overline{\overline{g}}\,\overline{\overline{q}}}-p_{g\,\overline{q}}||_{L^2(V)}.
\end{equation*}
Now, if we prove that $$||p_{\overline{\overline{g}}\,\overline{\overline{q}}}-p_{g\,\overline{q}}||_{L^2(V)}\leq\frac{1}{\lambda_0}||u_{\overline{\overline{g}}\,\overline{\overline{q}}}-u_{g\,\overline{q}}||_{\mathcal{H}}$$
the estimation (\ref{440}) holds. In fact, by the variational equality for the adjoint state given in \cite[see (5)]{TaBoGa}, for $g=\overline{\overline{g}}$ and $q=\overline{\overline{q}}$, we have
\begin{equation*}
 -\langle \dot{p}_{\overline{\overline{g}}\,\overline{\overline{q}}}(t),
v\rangle+a(p_{\overline{\overline{g}}\,\overline{\overline{q}}}(t),v)=(u_{\overline{\overline{g}}\,\overline{\overline{q}}}(t)-z_d(t),v)_H, \quad \forall
v\in V_0,
\end{equation*}
and for fixed $g\in\mathcal{H}$ and $q=\overline{q}$
\begin{equation*}
 -\langle \dot{p}_{g\,\overline{q}}(t),
v\rangle+a(p_{g\,\overline{q}}(t),v)=(u_{g\,\overline{q}}(t)-z_d(t),v)_H, \quad \forall
v\in V_0.
\end{equation*}
Subtracting these equations, we obtain
\begin{equation*}
 -\langle \dot{p}_{\overline{\overline{g}}\,\overline{\overline{q}}}(t)-\dot{p}_{g\,\overline{q}}(t),
v\rangle+a(p_{\overline{\overline{g}}\,\overline{\overline{q}}}(t)-p_{g\,\overline{q}}(t),v)=(u_{\overline{\overline{g}}\,\overline{\overline{q}}}(t)-u_{g\,\overline{q}}(t),v)_H, \quad \forall
v\in V_0.
\end{equation*}
Replacing $v=p_{\overline{\overline{g}}\,\overline{\overline{q}}}(t)-p_{g\,\overline{q}}(t)\in V_0$ and using that $$2\langle \dot{p}_{\overline{\overline{g}}\,\overline{\overline{q}}}(t)-\dot{p}_{g\,\overline{q}}(t),
p_{\overline{\overline{g}}\,\overline{\overline{q}}}(t)-p_{g\,\overline{q}}(t)\rangle=\frac{d}{dt}
||p_{\overline{\overline{g}}\,\overline{\overline{q}}}(t)-p_{g\,\overline{q}}(t)||^2_H,$$  we obtain
\begin{equation*}\begin{split}
&\quad-\frac{1}{2}\frac{d}{dt}
||p_{\overline{\overline{g}}\,\overline{\overline{q}}}(t)-p_{g\,\overline{q}}(t)||^2_H+\lambda_0
 ||\nabla\left(p_{\overline{\overline{g}}\,\overline{\overline{q}}}(t)-p_{g\,\overline{q}}(t)\right)||^2_H \\&\leq ||u_{\overline{\overline{g}}\,\overline{\overline{q}}}(t)-u_{g\,\overline{q}}(t)||_H||p_{\overline{\overline{g}}\,\overline{\overline{q}}}(t)-p_{g\,\overline{q}}(t)||_H.
\end{split}\end{equation*}
and by using Young inequality for $\epsilon=\lambda_0$, we have
\begin{equation*}\begin{split}
&\quad-\frac{1}{2}\frac{d}{dt}
||p_{\overline{\overline{g}}\,\overline{\overline{q}}}(t)-p_{g\,\overline{q}}(t)||^2_H+\lambda_0
 ||p_{\overline{\overline{g}}\,\overline{\overline{q}}}(t)-p_{g\,\overline{q}}(t)||^2_V\\&\leq \frac{1}{2\lambda_0}||u_{\overline{\overline{g}}\,\overline{\overline{q}}}(t)-u_{g\,\overline{q}}(t)||^2_H+\frac{\lambda_0}{2}||p_{\overline{\overline{g}}\,\overline{\overline{q}}}(t)-p_{g\,\overline{q}}(t)||_V^2.
\end{split}\end{equation*}
Then
\begin{small}\begin{equation*}\begin{split}
&\quad-\frac{d}{dt}
||p_{\overline{\overline{g}}\,\overline{\overline{q}}}(t)-p_{g\,\overline{q}}(t)||^2_H+\lambda_0
 ||p_{\overline{\overline{g}}\,\overline{\overline{q}}}(t)-p_{g\,\overline{q}}(t)||^2_V\leq \frac{1}{\lambda_0}||u_{\overline{\overline{g}}\,\overline{\overline{q}}}(t)-u_{g\,\overline{q}}(t)||^2_H.
\end{split}\end{equation*}\end{small}
By integrating between $0$ and $T$, and using that
$p_{\overline{\overline{g}}\,\overline{\overline{q}}}(T)=p_{g\,\overline{q}}(T)=0,$ we deduce
\begin{equation*}\begin{split}
&\quad
||p_{\overline{\overline{g}}\,\overline{\overline{q}}}(0)-p_{g\,\overline{q}}(0)||^2_H+\lambda_0
 ||p_{\overline{\overline{g}}\,\overline{\overline{q}}}-p_{g\,\overline{q}}||^2_{L^2(V)}\leq \frac{1}{\lambda_0}||u_{\overline{\overline{g}}\,\overline{\overline{q}}}-u_{g\,\overline{q}}||^2_{\mathcal{H}},
\end{split}\end{equation*}
i.e.
\begin{equation*}\begin{split}
&\quad
 ||p_{\overline{\overline{g}}\,\overline{\overline{q}}}-p_{g\,\overline{q}}||_{L^2(V)}\leq \frac{1}{\lambda_0}||u_{\overline{\overline{g}}\,\overline{\overline{q}}}-u_{g\,\overline{q}}||_{\mathcal{H}}.
\end{split}\end{equation*}
and then (\ref{440}) holds.
\end{proof}

Now, if we consider the distributed optimal control problem
\begin{equation}\label{distributed}
\text{find}\quad \overline{g}\in\mathcal{H} \quad\text{such that}\quad J_2(\overline{g})=\min\limits_{g\in\mathcal{H}}J_2(g)\qquad \text{for fixed}\quad q\in\mathcal{Q},
\end{equation}
where $J_2$ is the cost functional given in \cite{MT} plus the constant $\frac{M}{2}||q||^2_{\mathcal{Q}}$, that is,
$J_2:\mathcal{H}\rightarrow \mathbb{R}^+_0$ is given by
\begin{equation*}
J_2(g)=\frac{1}{2}||u_g-z_d||^2_{\mathcal{H}}+\frac{M_1}{2}||g||^2_{\mathcal{H}} +\frac{M}{2}||q||^2_{\mathcal{Q}}     \qquad (\text{fixed}\quad q\in \mathcal{Q}),
\end{equation*}
where $u_g$ is the unique solution of the problem (\ref{Pvariacional}) for fixed $q$, we can prove the following corollary.
\begin{cor}
If
$(\overline{\overline{g}},\overline{\overline{q}})\in\mathcal{H}\times\mathcal{Q}$
is the unique solution to the simultaneous optimal control problem studied in \cite[see (5)]{TaBoGa}, $\overline{g}$ is the unique solution to the distributed optimal control problem (\ref{distributed}) for fixed
$q$ ($q=\overline{\overline{q}}$), and  $\overline{q}$ is the unique solution to the problem (\ref{42}) for fixed $g$ ($g=\overline{\overline{g}}$) then $\overline{g}=\overline{\overline{g}}$ and
$\overline{q}=\overline{\overline{q}}$.
\end{cor}
\begin{proof}
If we take $h=\overline{\overline{g}}-\overline{g}$ in the optimality condition for the problem given in \cite{MT}, we have
\begin{equation}\label{1}
\left(M_1\overline{g}+p_{\overline{g}\,\overline{\overline{q}}},\overline{\overline{g}}-\overline{g}\right)_{\mathcal{H}}=0.
\end{equation}
On the other hand, if we consider $h=\overline{\overline{g}}-\overline{g}$ and $\eta=0$ in the optimality condition for the simultaneous optimal control problem studied in \cite{TaBoGa}, we obtain
\begin{equation}\label{2}
\left(M_1\overline{\overline{g}}+p_{\overline{\overline{g}}\,\overline{\overline{q}}},\overline{\overline{g}}-\overline{g}\right)_{\mathcal{H}}.
\end{equation}
Subtracting (\ref{1}) and (\ref{2}), we deduce that
$$\left(M_1\overline{g}+p_{\overline{g}\,\overline{\overline{q}}}-M_1\overline{\overline{g}}-p_{\overline{\overline{g}}\,\overline{\overline{q}}},\overline{\overline{g}}-\overline{g}\right)_{\mathcal{H}}=0, $$ and therefore
\begin{equation}\left(p_{\overline{g}\,\overline{\overline{q}}}-p_{\overline{\overline{g}}\,\overline{\overline{q}}},\overline{\overline{g}}-\overline{g}\right)_{\mathcal{H}}-M_1\left(\overline{\overline{g}}-\overline{g},\overline{\overline{g}}-\overline{g}\right)_{\mathcal{H}}=0.\nonumber \end{equation}
next, by using that
$$\left(p_{\overline{g}\,\overline{\overline{q}}}-p_{\overline{\overline{g}}\,\overline{\overline{q}}},\overline{\overline{g}}-\overline{g}\right)_{\mathcal{H}}=-||u_{\overline{g}\,\overline{\overline{q}}}-u_{\overline{\overline{g}}\,\overline{\overline{q}}}||^2_{\mathcal{H}},$$ we have
\begin{equation*}-||u_{\overline{g}\,\overline{\overline{q}}}-u_{\overline{\overline{g}}\,\overline{\overline{q}}}||^2_{\mathcal{H}}=M_1||\overline{\overline{g}}-\overline{g}||^2_{\mathcal{H}}.
\end{equation*}
Here, we deduce that $||\overline{\overline{g}}-\overline{g}||^2_{\mathcal{H}}=0$ and therefore
$\overline{g}=\overline{\overline{g}}$.
\newline In similar way we prove that
$\overline{q}=\overline{\overline{q}}$.
\end{proof}

\subsection{Estimations with respect to the problem $\mathbf{S_{\alpha}}$}
For each $\alpha>0$, we consider the following optimal control problem
\begin{equation}\label{42a}
\text{find}\quad \overline{q}_{\alpha}\in\mathcal{Q} \quad\text{such that}\quad J_{1\alpha}(\overline{q}_{\alpha})=\min\limits_{q\in\mathcal{Q}}J_{1\alpha}(q),
\end{equation}
where $J_{1\alpha}:\mathcal{Q}\rightarrow \mathbb{R}^+_0$ is given by
\begin{equation*}
J_{1\alpha}(q)=\frac{1}{2}||u_{\alpha q}-z_d||^2_{\mathcal{H}}+\frac{M_1}{2}||g||^2_{\mathcal{H}} +\frac{M}{2}||q||^2_{\mathcal{Q}}     \qquad (\text{fixed} \quad g\in \mathcal{H}),
\end{equation*}
that is, $J_{1\alpha}$ is the functional (\ref{PControlalfaJ}) plus the constant $\frac{M_1}{2}||g||^2_{\mathcal{H}}$, and $u_{\alpha q}$ is the unique solution of the problem (\ref{Palfavariacional}) for fixed $g$.
\begin{rem} The functional $J^{+}_{\alpha}$ defined in \cite [see (8)]{TaBoGa} and the functional $J_{1\alpha}$ previously defined satisfy the following estimate

\begin{equation*}
J^{+}_{\alpha}(\overline{\overline{g}}_{\alpha},\overline{\overline{q}}_{\alpha})\leq J_{1\alpha}(\overline{q}_{\alpha}),\quad \forall g\in\mathcal{H}.
\end{equation*}\end{rem}
Estimations between the solution of the Neumann boundary optimal control problem (\ref{42a}) with the second component of the solution to the simultaneous optimal control problem studied in \cite{TaBoGa}, is given in the following theorem whose prove is omitted.
\begin{thm}
If $(\overline{\overline{g}}_{\alpha},\overline{\overline{q}}_{\alpha})\in\mathcal{H}\times\mathcal{Q}$ is the unique solution of the simultaneous optimal control problem \cite[see (6)]{TaBoGa}, $\overline{q}_{\alpha}$ is the unique solutions to the optimal control problem (\ref{42a}), then we have:
\begin{equation*}\label{44}
||\overline{q}_{\alpha}-\overline{\overline{q}}_{\alpha}||_{\mathcal{Q}}\leq\frac{||\gamma_0||}{\lambda_{\alpha} M}||u_{\alpha \overline{\overline{g}}_{\alpha}\overline{\overline{q}}_{\alpha}}-u_{\alpha \,g\,\overline{q}_{\alpha}}||_{\mathcal{H}},\quad \forall g\in \mathcal{H}
\end{equation*}
with $\lambda_{\alpha}=\lambda_1\min\{1,\alpha\}$ and $\lambda_{1}$ the coerciveness constant of the bilinear form $a_{1}$.
\end{thm}
If we consider the following distributed optimal control problem, for each $\alpha>0$
\begin{equation}\label{distributedalfa}
\text{find}\quad \overline{g}_{\alpha}\in\mathcal{H} \quad\text{such that}\quad J_{2\alpha}(\overline{g}_{\alpha})=\min\limits_{g\in\mathcal{H}}J_{2\alpha}(g),
\end{equation}
where $J_{2\alpha}:\mathcal{H}\rightarrow \mathbb{R}^+_0$ is given by
\begin{equation*}
J_{2\alpha}(g)=\frac{1}{2}||u_{\alpha g}-z_d||^2_{\mathcal{H}}+\frac{M_1}{2}||g||^2_{\mathcal{H}} +\frac{M}{2}||q||^2_{\mathcal{Q}}     \qquad (\text{fixed} \quad q\in \mathcal{Q}),
\end{equation*}
that is, $J_{2\alpha}$ is the functional studied in \cite{MT} plus the constant $\frac{M}{2}||q||^2_{\mathcal{Q}}$, and $u_{\alpha g}$ is the unique solution of the problem (\ref{Palfavariacional}) for fixed $q$, we give the following corollary, whose prove is omitted.
\begin{cor}
If
$(\overline{\overline{g}}_{\alpha},\overline{\overline{q}}_{\alpha})\in\mathcal{H}\times\mathcal{Q}$
is the unique solution of the simultaneous optimal control problem studied in \cite[see (6)]{TaBoGa}, $\overline{g}_{\alpha}$ is the unique solution of the problem (\ref{distributedalfa}) for fixed
$q$ ($q=\overline{\overline{q}}_{\alpha}$), and  $\overline{q}_{\alpha}$ is the unique solution of the problem (\ref{42a}) for fixed $g$ ($g=\overline{\overline{g}}_{\alpha}$) then
$\overline{g}_{\alpha}=\overline{\overline{g}}_{\alpha}$ and
$\overline{q}_{\alpha}=\overline{\overline{q}}_{\alpha}$.\end{cor}

\section{Real Neumann Boundary Optimal Control \linebreak Problems}

In this Section, we consider the non-stationary real-boundary optimal control problems (\ref{PControlH}) and (\ref{PControlalfaH}) and the stationary real-boundary optimal control problems (\ref{PControlHEstacionario}) and (\ref{PControlalfaHEstacionario}). We prove existence and uniqueness of the solutions to these optimal control problems and monotonicity results are also obtained.

\subsection{Real Neumann boundary optimal control problem in relation to the parabolic system $\mathbf{S}$}

If we consider the real-boundary optimal control problem (\ref{PControlH}) and we denote by $u_{bqg}$ to the unique solution of the variational equality  (\ref{Pvariacional}) for data $b$, $q$ and $g$ and we take $q=\lambda q_{0}$ for fixed $q_{0}\in \mathcal{Q}$ $(q_{0}\neq 0)$ and $\lambda \in \mathbb{R}$, we can prove that
\[
u_{b q g}(t)=u_{b\lambda g}(t)=u_{b}(t)+u_{q}(t)+u_{g}(t),\quad \forall x\in \Omega.
\]
where $u_{b}$ is the unique solution to the parabolic variational equality
\begin{equation}\label{Pvariacionalb}
\left\{
\begin{array}{l l}
u-v_b \in L^2(V_0), \qquad u(0)=v_b\quad\text{and}\quad \dot{u}\in L^2( V_0')\\
\text{such that}\quad \langle \dot{u}(t), v\rangle+a(u(t),v)=0,
\quad \forall v\in V_0,
\end{array}
\right.\end{equation}
$u_{q}$ is the unique solution to the parabolic variational equality
\begin{equation}\label{Pvariacionalq}
\left\{
\begin{array}{l l}
u\in L^2(V_0), \qquad u(0)=0\quad\text{and}\quad \dot{u}\in L^2( V_0')\\
\text{such that}\quad \langle \dot{u}(t), v\rangle+a(u(t),v)=-\lambda(q_0(t),v)_Q,
\quad \forall v\in V_0,
\end{array}
\right.\end{equation}
and $u_{g}$ is the unique solution to the parabolic variational equality
\begin{equation}\label{Pvariacionalg}
\left\{
\begin{array}{l l}
u\in L^2(V_0), \qquad u(0)=0\quad\text{and}\quad \dot{u}\in L^2( V_0')\\
\text{such that}\quad \langle \dot{u}(t), v\rangle+a(u(t),v)=(g(t),v)_H,
\quad \forall v\in V_0.
\end{array}
\right.\end{equation}
We note that, by the linearity, we can prove that $u_{q}(t)=\lambda u_{q_{0}}(t)$, where $u_{q_{0}}$ is the unique solution of (\ref{Pvariacional}) for $q=q_{0}$ and $b=g=0$.
\newline
Next, for each $T>0$, the functional $H_{T}(\lambda)$ can be writen as
\begin{equation*} H_{T}(\lambda)=\frac{1}{2}\int\limits_0^T \int\limits_{\Omega} (u_{b}(t)
+\lambda u_{q_0}(t)+u_{ g}(t)-z_d(t))^2 dx dt+\frac{M\lambda
^2}{2}\int\limits_0^T \int\limits_{\Gamma_{2}} q_0^2(t) d\gamma dt
\end{equation*}
therefore, $H_{T}(\lambda)=\lambda^2 A(T)+\lambda B(T) +C(T)$, where
$$A(T)=\frac{M}{2}\int\limits_0^T \int\limits_{\Gamma_{2}} q_0^2(t) d\gamma dt+\frac{1}{2}\int\limits_0^T \int\limits_{\Omega} u^2_{q_0}(t)dxdt$$
$$B(T)=\int\limits_0^T \int\limits_{\Omega} u_{q_0}(t)( u_{b}(t)+u_{ g}(t)-z_d(t))dx dt$$
$$C(T)=\frac{1}{2}\int\limits_0^T \int\limits_{\Omega}( u_{b}(t)+u_{ g}(t)-z_d(t))^2 dxdt.$$
Here, taking into account that
\begin{small}\begin{equation*}\begin{split}
&  \quad  4 A(T)C(T)\\ & =\left(M\int\limits_0^T \int\limits_{\Gamma_{2}} q_0^2(t) d\gamma dt+\int\limits_0^T \int\limits_{\Omega} u^2_{q_0}(t)dxdt\right)\left(\int\limits_0^T \int\limits_{\Omega}( u_{b}(t)+u_{ g}(t)-z_d(t))^2 dxdt\right)\\ & > \left(\int\limits_0^T \int\limits_{\Omega} u^2_{q_0}(t)dxdt \right)\left(\int\limits_0^T\int\limits_{\Omega}( u_{b}(t)+u_{ g}(t)-z_d(t))^2 dxdt \right)\\ &= || u_{q_{0}}||^{2}_{\mathcal{H}} || u_{b}+u_{g}-z_{d}||^{2}_{\mathcal{H}}\geq (u_{q_{0}},u_{b}+u_{g}-z_{d})^{2}_{\mathcal{H}}\\ & =\left(\int\limits_0^T \int\limits_{\Omega} u_{q_0}(t)( u_{b}(t)+u_{ g}(t)-z_d(t))dx dt\right)^{2}=(B(T))^{2}
\end{split}\end{equation*}\end{small}
we deduce that $(B(T))^{2}-4 A(T)C(T)<0$ and since $A(T)>0$, because $q_0\neq 0$, then there exists a unique $\overline{\lambda}(T)\in \mathbb{R}$ for each $T>0$, such that satisfies the problem (\ref{PControlH}), whose solution is given by the following expression:
\begin{equation}\label{lambda1}
\overline{\lambda}(T)=-\frac{B(T)}{2A(T)}=-\frac{\int\limits_0^T \int\limits_{\Omega}
u_{q_0}(t)( u_{b}(t) +u_{ g}(t)-z_d(t))dx dt}{M\int\limits_0^T
\int\limits_{\Gamma_{2}} q_0^2(t) d\gamma dt+\int\limits_0^T
\int\limits_{\Omega} u^2_{q_0}(t)dxdt}.
\end{equation}
Therefore, we have proved the following property.
\begin{thm}
For each $T>0$, there exists a unique solution $\overline{\lambda}(T)\in \mathbb{R}$ to the optimization problem (\ref{PControlH}).
\end{thm}
Now, we will prove some monotonicity properties.
\begin{prop}\label{Prop2}
Let $q_{1}=\lambda_{1} q_{0}$ and $q_{2}=\lambda_{2} q_{0}$ ($q_{0}>0$), with $\lambda_{2}\leq \lambda_{1}$ and $g_{1}\leq g_{2}$ then
$ u_{b\lambda_1g_{1}}\leq u_{b\lambda_2g_{2}}$ in $\Omega\times [0,T]$.
\end{prop}
\begin{proof}
We define $w=u_{b\lambda_1
g_{1}}-u_{b\lambda_2 g_{2}}$ and we take $v=-w^+(t)\in V$  for $u_{b\lambda_1g_{1}}$ in (\ref{Pvariacional}) and $v=w^+(t)\in V$ for $u_{b\lambda_2 g_{2}}$ in (\ref{Pvariacional}) (for regularity of $w^+$ see \cite{KS}). Adding the variational equalities, we have
\begin{equation*}\begin{split}
 \langle \dot{u}_{b\lambda_2 g_2}(t)-\dot{u}_{b\lambda_1 g_1}(t),w^+(t)\rangle+a(u_{b\lambda_2 g_2}(t)-u_{b\lambda_1 g_1}(t),w^+(t))
 \\ =  \int_{\Omega}(g_{2}(t)-g_{1}(t)) w^{+}(t) dx +\left(\lambda_1-\lambda_2\right)\int\limits_{\Gamma_2}q_0(t)w^+(t)
d\gamma.\end{split}
\end{equation*}
next
\begin{equation*}\begin{split}
 - \langle\dot{w^+}(t),w^+(t)\rangle-a(w^+(t),w^+(t))
&=\int\limits_{\Omega}(g_{2}(t)-g_{1}(t)) w^{+}(t) dx \\ &+\left(\lambda_1-\lambda_2\right)\int\limits_{\Gamma_2}q_0(t)w^+(t)
d\gamma.
\end{split}\end{equation*}
Now, using that
$\langle\dot{w^+}(t),w^+(t)\rangle=\frac{1}{2}\frac{d}{dt}||w^+(t)||_H^2$ and integrating between $0$ and $T$, we prove
\begin{small}\begin{equation*}\begin{split}
\frac{1}{2}\left(||w^+(T)||_H^2-||w^+(0)||_H^2\right) +\int\limits_0^T ||w^+(t)||^2_{V_0}dt
&\leq \int\limits_0^T \int\limits_{\Omega}(g_{1}(t)-g_{2}(t)) w^{+}(t) dx dt\\& +\left(\lambda_2-\lambda_1\right)\int\limits_0^T\int\limits_{\Gamma_2}q_0(t)w^+(t)
d\gamma dt.
\end{split}\end{equation*}\end{small}
Since $w^+(0)=\left(u_{b\lambda_1 g_1}-u_{b\lambda_2
g_2}\right)^+(0)=\max\{0,\left(u_{b\lambda_1 g_1}-u_{b\lambda_2
g_2}\right)(0)\}=0$,
\begin{equation*}\begin{split}
 \frac{1}{2}||w^+(T)||_H^2 +\int\limits_0^T ||w^+(t)||^2_{V_0}dt
& \leq \int\limits_0^T \int\limits_{\Omega}(g_{1}(t)-g_{2}(t)) w^{+}(t) dx dt\\& +\left(\lambda_2-\lambda_1\right)\int\limits_0^T\int\limits_{\Gamma_2}q_0(t)w^+(t)
d\gamma dt\leq 0,
\end{split}\end{equation*}
by using that $g_{1}\leq g_{2}$, $\lambda_{2}\leq \lambda_{1}$ and $q_{0}>0$, here $||w^+||^2_{L^2(0,T;V_0)}
=0$, then $w^+\equiv 0$ in $\Omega\times [0,T]$ and therefore $w\leq 0$ in $\Omega\times [0,T]$, that is the thesis holds.
\end{proof}
\begin{cor} \label{Prop1} If $q_{1}=\lambda_{1} q_{0}$ and $q_{2}=\lambda_{2} q_{0}$ ($q_{0}>0$), with $\lambda_{2}\leq \lambda_{1}$ then $ u_{b\lambda_1g}\leq u_{b\lambda_2g}$ in $\Omega\times [0,T]$.
\end{cor}
\begin{proof}
This results taking $g=g_{1}=g_{2}$ in the proof of the Proposition \ref{Prop2}.
\end{proof}

\begin{rem}
The previous monotonicity properties are still true if we consider $\lambda_{1}\leq \lambda_{2}$ with $q_{0}<0$.
\end{rem}

\subsection{Real Neumann boundary optimal control problems in relation to the parabolic $\mathbf{S_{\alpha}}$}

If we consider the real Neumann boundary optimal control problem (\ref{PControlalfaH}) and for each $\alpha >0$, we denote by $u_{\alpha bqg}$ to the unique solution of the variational equality  (\ref{Palfavariacional}) for data $b$, $q$ and $g$ and we take $q=\lambda q_{0}$ for fixed $q_{0}\in Q$ $(q_{0}\neq 0)$ and $\lambda \in \mathbb{R}$, we can see that
\[
u_{\alpha b q g}(t)=u_{\alpha b\lambda g}(t)=u_{\alpha b}(t)+u_{\alpha q}(t)+u_{\alpha g}(t),\quad \forall x\in \Omega.
\]
where $u_{\alpha b},u_{\alpha q},u_{\alpha g}$ are the unique solution of (\ref{Palfavariacional}) for $q=g=0$, $b=g=0$ and $b=q=0$, respectively.

We note that, by the linearity, we can see that $u_{\alpha q}(t)=\lambda u_{\alpha q_{0}}(t)$, where $u_{\alpha q_{0}}$ is the unique solution of (\ref{Palfavariacional}) for $q=q_{0}$ and $b=g=0$. Next, for each $T>0$, the functional $H_{\alpha T}(\lambda)$ can be writen as
\begin{equation*} \begin{split}
H_{\alpha T}(\lambda)&=\frac{1}{2}\int\limits_0^T \int\limits_{\Omega} (u_{\alpha b}(t)
+\lambda u_{\alpha q_0}(t)+u_{\alpha g}(t)-z_d(t))^2 dx dt\\ & +\frac{M\lambda
^2}{2}\int\limits_0^T \int\limits_{\Gamma_{2}} q_0^2(t) d\gamma dt=\lambda^2 A_{\alpha}(T)+\lambda B_{\alpha}(T) +C_{\alpha}(T),
\end{split}\end{equation*}
where
$$A_{\alpha}(T)=\frac{M}{2}\int\limits_0^T \int\limits_{\Gamma_{2}} q_0^2(t) d\gamma dt+\frac{1}{2}\int\limits_0^T \int\limits_{\Omega} u^2_{\alpha q_0}(t)dxdt$$
$$B_{\alpha}(T)=\int\limits_0^T \int\limits_{\Omega} u_{\alpha q_0}(t)( u_{\alpha b}(t)+u_{\alpha g}(t)-z_d(t))dx dt$$
$$C_{\alpha}(T)=\frac{1}{2}\int\limits_0^T \int\limits_{\Omega}( u_{\alpha b}(t)+u_{\alpha g}(t)-z_d(t))^2 dxdt.$$
Now, taking into account that $A_{\alpha}(T)>0$, in similar way that the previous subsection, we can prove that $(B_{\alpha}(T))^{2}-4 A_{\alpha}(T)C_{\alpha}(T)<0$ and therefore there exists a unique solution $\overline{\lambda}_{\alpha}(T)\in \mathbb{R}$, for each $\alpha >0$ and for each $T>0$, for the optimal control problem (\ref{PControlalfaH}), which is given by the following expression:
\begin{equation}\label{lambda3}
\overline{\lambda}_{\alpha}(T)=-\frac{\int\limits_0^T \int\limits_{\Omega}
u_{\alpha q_0}(t)( u_{\alpha b}(t) +u_{\alpha g}(t)-z_d(t))dx
dt}{M\int\limits_0^T \int\limits_{\Gamma_{2}} q_0^2(t) d\gamma
dt+\int\limits_0^T \int\limits_{\Omega} u^2_{\alpha q_0}(t)dxdt}.
\end{equation}
Therefore, we have proved the following property.
\begin{thm}
For each $\alpha >0$ and $T>0$, there exists a unique solution $\overline{\lambda}_{\alpha}(T)\in \mathbb{R}$ to the optimization problem (\ref{PControlalfaH}).
\end{thm}
Now, in similar way to the previous subsection, we can prove the following monotonicity properties, whose proof is omitted.
\begin{prop} \label{Propalfa1}
For each $\alpha >0$, if $q_1=\lambda_1 q_0$ and $q_2=\lambda_2 q_0$ ($q_{0}>0$), with $\lambda_2\leq \lambda_1,$ $g_1\leq g_2,$ $b_1\leq b_2$ on $\Gamma_1$ and initial conditions $v_{b_1}\leq v_{b_2}$ then $u_{\alpha b_1\lambda_1g_1}\leq u_{\alpha b_2\lambda_2g_2}$ in $\Omega\times [0,T]$.
\end{prop}

\subsection{Real Neumann boundary optimal control problem in relation to elliptic system $\mathbf{P}$}

Here, we consider the stationary real Neumann boundary optimal control problem (\ref{PControlHEstacionario}). If we denote by $u_{\infty bqg}$ to the unique solution of the variational equality (\ref{PvariacionalEstacionario}) for data $b$, $q$ and $g$. and if we consider $q=\lambda q^{*}_{0}$ for fixed $q^{*}_{0}\in Q$ $(q^{*}_{0}\neq 0)$ and $\lambda \in \mathbb{R}$, we can see that
\[
u_{\infty b q g}=u_{\infty b\lambda g}=u_{\infty b}+u_{\infty q}+u_{\infty g},\quad \forall x\in \Omega.
\]
where $u_{\infty b},u_{\infty q},u_{\infty g}$ are the unique solutions of the variational equality (\ref{PvariacionalEstacionario}) for $q=g=0$, $b=g=0$ and $b=q=0$, respectively.
\newline
Now, taking into account that $u_{\infty q}=\lambda u_{\infty q^{*}_{0}}$, where $u_{\infty q^{*}_{0}}$ is the solution of (\ref{PvariacionalEstacionario}) for $q=q^{*}_{0}$ and $b=g=0$. Next, the functional $H(\lambda)$ can be writen as
\begin{equation*} H(\lambda)=\frac{1}{2} \int\limits_{\Omega} (u_{\infty b}
+\lambda u_{\infty q^{*}_0}+u_{\infty g}-z_d)^2 dx +\frac{M\lambda
^2}{2}\int\limits_{\Gamma_{2}} (q^{*}_0)^2 d\gamma.
\end{equation*}
Therefore $H(\lambda)=\lambda^{2}A+\lambda B+C$, where
$$A=\frac{M}{2} \int\limits_{\Gamma_{2}} (q^{*}_0)^2 d\gamma +\frac{1}{2} \int\limits_{\Omega} u^2_{\infty q^{*}_{0}} dx,\quad B=\int\limits_{\Omega} u_{\infty q^{*}_{0}}(u_{\infty b}+u_{\infty g}-z_d)dx$$
$$C=\frac{1}{2}\int\limits_{\Omega}(u_{\infty b}+u_{\infty g}-z_d)^2 dx.$$
Here, since $B^{2}-4AC<0$, there exists a unique $\overline{\lambda}_{\infty}\in\mathbb{R}$ such that satisfies the problem (\ref{PControlHEstacionario}), that is
\begin{equation}\label{lambda2}
\overline{\lambda}_{\infty}=-\frac{B}{2A}=-\frac{\int\limits_{\Omega}
u_{\infty q^{*}_{0}}( u_{\infty b} +u_{\infty g}-z_d)dx}{M
\int\limits_{\Gamma_{2}} (q^{*}_0)^2 d\gamma +\int\limits_{\Omega} u^2_{\infty q^{*}_0}dx}
\end{equation}
Therefore, we have proved the following theorem.
\begin{thm}
There exists a unique solution $\overline{\lambda}_{\infty}\in \mathbb{R}$ to the optimization problem (\ref{PControlHEstacionario}).
\end{thm}
Now, we will give some monotonicity properties whose proof is omitted.
\begin{prop}
Let $q_{1}=\lambda_{1} q^{*}_{0}$ and $q_{2}=\lambda_{2} q^{*}_{0}$ ($q^{*}_{0}>0$), with $\lambda_2\leq \lambda_1$ and $g_{1}\leq g_{2}$ then
$ u_{\infty b\lambda_1g_{1}}\leq u_{\infty b\lambda_2g_{2}}$.
\end{prop}

\subsection{Real Neumann boundary optimal control problems in relation to the elliptic system $\mathbf{P_{\alpha}}$}

We consider the stationary real Neumann boundary optimal control problem (\ref{PControlalfaHEstacionario}) and we denote  by $u_{\infty\alpha bqg}$ to the unique solution to the variational equality (\ref{PalfavariacionalEstacionario}) for data $b$, $q$ and $g$. If we consider $q=\lambda q^{*}_{0}$ for fixed $q^{*}_{0}\in Q$ $(q^{*}_{0}\neq 0)$ and $\lambda \in \mathbb{R}$, we can see that
\[
u_{\infty\alpha b q g}=u_{\infty\alpha b\lambda g}=u_{\infty\alpha b}+u_{\infty\alpha q}+u_{\infty\alpha g},\quad \forall x\in \Omega.
\]
where $u_{\infty\alpha b},u_{\infty\alpha q},u_{\infty\alpha g}$ are the unique solutions of the variational equality (\ref{PalfavariacionalEstacionario}) for $q=g=0$, $b=g=0$ and $b=q=0$, respectively.
\newline
Now, taking into account that $u_{\infty\alpha q}=\lambda u_{\infty\alpha q^{*}_{0}}$, where $u_{\infty\alpha q^{*}_{0}}$ is the solution of (\ref{PalfavariacionalEstacionario}) for $q=q^{*}_{0}$ and $b=g=0$, the functional $H_{\alpha}(\lambda)$ can be writen as
\begin{small}\begin{equation*} H_{\alpha}(\lambda)=\frac{1}{2} \int\limits_{\Omega} (u_{\infty\alpha b}
+\lambda u_{\infty\alpha q^{*}_0}+u_{\infty\alpha g}-z_d)^2 dx +\frac{M\lambda
^2}{2}\int\limits_{\Gamma_{2}} (q^{*}_0)^2 d\gamma=\lambda^{2}A_{\alpha}+\lambda B_{\alpha}+C_{\alpha},
\end{equation*}\end{small}
where
$$A_{\alpha}=\frac{M}{2} \int\limits_{\Gamma_{2}} (q^{*}_0)^2 d\gamma +\frac{1}{2} \int\limits_{\Omega} u^2_{\infty\alpha q^{*}_0} dx,\quad B_{\alpha}=\int\limits_{\Omega} u_{\infty\alpha q^{*}_0}(u_{\infty\alpha b}+u_{\infty\alpha g}-z_d)dx$$
$$C_{\alpha}=\frac{1}{2}\int\limits_{\Omega}(u_{\infty\alpha b}+u_{\infty\alpha g}-z_d)^2 dx.$$
Here, since $B_{\alpha}^{2}-4A_{\alpha}C_{\alpha}<0$, there exists a unique $\overline{\lambda}_{\alpha}\in \mathbb{R}$ such that satisfies the optimization problem (\ref{PControlalfaHEstacionario}), that is
\begin{equation}\label{lambda2alfa}
\overline{\lambda}_{\alpha}=-\frac{B_{\alpha}}{2A_{\alpha}}=-\frac{\int\limits_{\Omega}
u_{\infty\alpha q^{*}_0}( u_{\infty\alpha b} +u_{\infty\alpha g}-z_d)dx}{M
\int\limits_{\Gamma_{2}} (q^{*}_0)^2 d\gamma +\int\limits_{\Omega} u^2_{\infty\alpha q^{*}_0}dx}.
\end{equation}
Therefore, we have proved the following theorem.
\begin{thm}
For each $\alpha >0$, there exists a unique solution $\overline{\lambda}_{\alpha}\in \mathbb{R}$ to the optimization problem (\ref{PControlalfaHEstacionario}).
\end{thm}
Now, we will give some monotonicity properties whose proof is omitted.
\begin{prop}
For each $\alpha >0$, if $q_1=\lambda_1 q^{*}_0$ and  $q_2=\lambda_2 q^{*}_0$ ($q^{*}_{0}>0$), with $\lambda_2\leq \lambda_1$, $g_1\leq g_2$, $b_1\leq b_2$ on $\Gamma_1$ then $u_{\alpha b_1\lambda_1g_1}\leq u_{\alpha b_2\lambda_2g_2}$ in $\Omega$.
\end{prop}

\section{Asymptotic behaviour of the solutions when \linebreak $\mathbf{t\rightarrow +\infty}$}

In this Section, we study the convergence of the solutions of the problem (\ref{Pvariacional}) for fixed data $b\in H^{\frac{1}{2}}(\Gamma_{1})$, $q\in \mathcal{Q}$ and $g\in \mathcal{H}$ to the solution to the problem (\ref{PvariacionalEstacionario}) for the same $b\in H^{\frac{1}{2}}(\Gamma_{1})$ and fixed $q\in Q$ and $g\in H$, when $t\rightarrow +\infty$. Here, for the sake of simplicity, we denote by $u_{\infty}$ to the unique solution to the variational equality (\ref{PvariacionalEstacionario}) for data $q_{\infty}\in Q$ and  $g_{\infty}\in H$.

If we define
\[
F_{1}(t)= e^{\lambda_{0}t}||g(t)-g_{\infty}||^{2}_{H}, \qquad
F_{2}(t)= e^{\lambda_{0}t}||\gamma_{0}||^{2}||q(t)-q_{\infty}||^{2}_{Q}
\]
with $g\in \mathcal{H}$, $q\in \mathcal{Q}$, $\lambda_{0}$ the coerciveness constant of the bilinear form $a$ and $\gamma_{0}$ the trace operator, we can prove the following theorem.

\begin{thm}\label{asymptotic}
If $b\in H^{\frac{1}{2}}(\Gamma_{1})$, $q\in \mathcal{Q}$, $g\in \mathcal{H}$, $F_{1}\in L^{1}(0,\infty)$ and $F_{2}\in L^{1}(0,\infty)$, then
\begin{small}\begin{equation}
||u_{bqg}(t)-u_{\infty}||^{2}_{H} \leq ||u_{bqg}(0)-u_{\infty}||^{2}_{H}e^{-\lambda_{0}t}+ \frac{2e^{-\lambda_{0}t}}{\lambda_{0}}\left(||F_{1}||_{L^{1}(0,\infty)}+||F_{2}||_{L^{1}(0,\infty)}\right)\nonumber
\end{equation}\end{small}
and therefore
\begin{equation}
\lim_{t\rightarrow +\infty}u_{bqg}(t)=u_{\infty}\quad \text{in}\,\, H \,\, \text{strong (exponentially)}.\nonumber
\end{equation}
\end{thm}
\begin{proof}
If we consider $w(t)=u_{bqg}(t)-u_{\infty}$, we have that $w(t)\in V_{0}$, $w(0)=v_{b}-u_{\infty}$ and $\dot{w}=\dot{u}_{bqg}$. Therefore, taking $v=w(t)$ in the variational equalities (\ref{Pvariacional}) and  (\ref{PvariacionalEstacionario}) respectively, and subtracting then we obtain
\begin{equation*}
\langle \dot{w}(t), w(t)\rangle+a(w(t),w(t))=\langle L(t)-L_{\infty},w(t)\rangle
\end{equation*}
that is
\begin{equation*}
\frac{1}{2}\frac{d}{dt}||w(t)||^{2}_{H}+\lambda_{0}||w(t)||^{2}_{V}\leq(g(t)-g_{\infty},w(t))_{H}-(q(t)-q_{\infty},w(t))_{Q}.
\end{equation*}
next
\begin{small}\begin{equation*}\begin{split}
\frac{1}{2}\frac{d}{dt}||w(t)||^{2}_{H}+\lambda_{0}||w(t)||^{2}_{V}& \leq ||g(t)-g_{\infty}||_{H} ||w(t)||_{V}+ ||q(t)-q_{\infty}||_{Q} ||w(t)||_{Q}\\& \leq \left(||g(t)-g_{\infty}||_{H} + ||\gamma_{0}|| ||q(t)-q_{\infty}||_{Q} \right)||w(t)||_{V}\\ &\leq \frac{\lambda_{0}}{2}||w(t)||^{2}_{V}\\& \quad +\frac{1}{2\lambda_{0}}\left(||g(t)-g_{\infty}||_{H} + ||\gamma_{0}|| ||q(t)-q_{\infty}||_{Q} \right)^{2}\end{split}
\end{equation*}\end{small}
where $\gamma_{0}$ is the trace operator. Then
\begin{equation*}\begin{split}
\frac{1}{2}\frac{d}{dt}||w(t)||^{2}_{H}+\frac{\lambda_{0}}{2}||w(t)||^{2}_{H}& \leq \frac{1}{\lambda_{0}}\left(||g(t)-g_{\infty}||^{2}_{H} + ||\gamma_{0}||^{2} ||q(t)-q_{\infty}||^{2}_{Q} \right)\end{split}
\end{equation*}
next, if we call $F(t)=||g(t)-g_{\infty}||^{2}_{H} + ||\gamma_{0}||^{2} ||q(t)-q_{\infty}||^{2}_{Q}$, we have
\begin{equation*}\begin{split}
\frac{d}{dt}\left(||w(t)||^{2}_{H}\right) e^{\lambda_{0}t}+\lambda_{0}||w(t)||^{2}_{H} e^{\lambda_{0}t} & \leq \frac{2}{\lambda_{0}}F(t) e^{\lambda_{0}t}\end{split}
\end{equation*}
or equivalently
\begin{equation*}
\frac{d}{dt}\left(||w(t)||^{2}_{H} e^{\lambda_{0}t}\right) \leq \frac{2}{\lambda_{0}}F(t) e^{\lambda_{0}t}.
\end{equation*}
Now, integrating between $0$ and $t$, we have
\begin{equation*}
||w(t)||^{2}_{H} e^{\lambda_{0}t}-||w(0)||^{2}_{H}\leq \frac{2}{\lambda_{0}}\int_{0}^{t}F(\tau) e^{\lambda_{0}\tau} d\tau
\end{equation*}
therefore
\begin{equation*}
||w(t)||^{2}_{H} \leq ||w(0)||^{2}_{H}e^{-\lambda_{0}t}+ \frac{2e^{-\lambda_{0}t}}{\lambda_{0}}\int_{0}^{t}F(\tau) e^{\lambda_{0}\tau} d\tau
\end{equation*}
and the thesis holds.
\end{proof}

\begin{cor}
If $b\in H^{\frac{1}{2}}(\Gamma_{1})$, $F_{1}\in L^{1}(0,\infty)$ and $F_{2}\in L^{1}(0,\infty)$, with $g(t)=g_{\infty}\in H$ and $q(t)=q_{\infty}\in Q$ then we have
\begin{equation}
||u_{bqg}(t)-u_{\infty}||_{H} \leq ||u_{bqg}(0)-u_{\infty}||_{H}e^{-\frac{\lambda_{0}}{2}t}.
\end{equation}
\end{cor}
\begin{proof}
This results directly from Theorem \ref{asymptotic}.
\end{proof}

\begin{rem}
We note that, if we consider the hypotheses: there exist $m\in (0,\lambda_{0})$ and $C_{1}=const.>0$ such that
\[
\lim_{t\rightarrow +\infty}\frac{||F_{1}||_{L^{1}(0,t)}+||F_{2}||_{L^{1}(0,t)}}{e^{mt}}\leq C_{1}
\]
we can prove the asymptotic behaviour obtained in Theorem \ref{asymptotic}.
\end{rem}
\begin{cor}
\begin{itemize}
\item [a)] If $b\in H^{\frac{1}{2}}(\Gamma_{1})$, with  $q(t)=q_{\infty}\in Q$ and there exist $ m\in (0,\lambda_{0})$ and $C_{2}=const.>0$ such that $\lim\limits_{t\rightarrow +\infty}\frac{||F_{1}||_{L^{1}(0,t)}}{e^{mt}}\leq C_{2}$, then
\[
\lim_{t\rightarrow +\infty}||u_{bqg}(t)-u_{\infty}||_{H}=0.
\]
\item [b)] If $b\in H^{\frac{1}{2}}(\Gamma_{1})$, with $g(t)=g_{\infty}\in H$ and there exist $ m\in (0,\lambda_{0})$ and $C_{3}=const.>0$ such that $\lim\limits_{t\rightarrow +\infty}\frac{||F_{2}||_{L^{1}(0,t)}}{e^{mt}}\leq C_{3}$, then
\[
\lim_{t\rightarrow +\infty}||u_{bqg}(t)-u_{\infty}||_{Q}=0.
\]
\end{itemize}
\end{cor}

\begin{rem}
An open problem is to study if (\ref{lambda1}) is convergent to (\ref{lambda2}) and (\ref{lambda3}) is convergent to (\ref{lambda2alfa}) (for each $\alpha >0$),  when $T\rightarrow +\infty$. We hope that these convergences do not happen. This assumption is based on the fact that a function $g(t)$ ($\forall t >0$) can be strongly convergent in $H$ to a function $g_{\infty}$, but not necessarily $g$ is strongly convergent to the same function $g_{\infty}$ in $\mathcal{H}$, which is showed in the following counterexample.
\end{rem}
\begin{exa}
Let $\Omega=(0,1)$ be and $g(t)=g_{\infty}+e^{-t}$, we have that
\[
\int_{0}^{1}(g(t)-g_{\infty})^{2}dx=\int_{0}^{1}(e^{-t})^{2}dx=e^{-2t}\rightarrow 0,\quad \text{if} \,\,\, t\rightarrow +\infty,
\]
and
\[
\int_{0}^{t}\int_{0}^{1}(g(\tau)-g_{\infty})^{2}dx d\tau=\int_{0}^{t}(e^{-\tau})^{2}d\tau=\frac{1}{2}(1-e^{-2t})\rightarrow \frac{1}{2},\quad \text{if} \,\,\,  t\rightarrow +\infty.
\]
\end{exa}

\section*{Acknowledgements}

The present work has been partially sponsored by the European Union’s Horizon 2020 Research and Innovation Programme under the Marie Sklo\-dowska-Curie grant agreement 823731 CONMECH for the third author; by the Project PIP No. 0275 from CONICET and Universidad Austral, Rosario, Argentina for the second and third authors; and by the Project PPI No. 18/C555 from SECyT-UNRC, Río Cuarto, Argentina for the first and second authors. The authors thank the anonymous referee whose comments helped us to improve our paper.


\end{document}